\documentclass[12pt]{amsart}

\usepackage{amscd,latexsym,amsthm,amsfonts,amssymb,amsmath,amsxtra}
\usepackage[mathscr]{eucal}
\usepackage{graphicx}
\usepackage{float}
\usepackage{hyperref}
\renewcommand\theequation{\thesection.\arabic{equation}}

\newcommand{\BA}{{\mathbb {A}}}

\newcommand{\BC}{{\mathbb {C}}}

\newcommand{\BZ}{{\mathbb {Z}}}

\newcommand{\CH}{{\mathcal {H}}}

\newcommand{\Ad}{{\mathrm{Ad}}}

\newcommand{\GL}{{\mathrm{GL}}}

\newcommand{\Mat}{{\mathrm{Mat}}}

\newcommand{\Sp}{{\mathrm{Sp}}}
\newcommand{\Stab}{{\mathrm{Stab}}}

\newcommand{\tr}{{\mathrm{tr}}}

\newcommand{\ovl}{\overline}
\newcommand{\udl}{\underline}
\newcommand{\wt}{\widetilde}

\newcommand{\ol}{\overline}
\newcommand{\ul}{\underline}

\newcommand{\bs}{\backslash}

\def\bks{{\backslash}}

\def\diag{{\rm diag}}

\def\tilsp{{\widetilde{\Sp}}}

\newtheorem{thm}{Theorem}[section]
\newtheorem{cor}[thm]{Corollary}
\newtheorem{lem}[thm]{Lemma}
\newtheorem{prop}[thm]{Proposition}

\newtheorem {ques/conj}[thm]{Question/Conjecture}

\newtheorem{rmk}[thm]{Remark}

\newtheorem{exmp}[thm]{Example}

\makeatletter

\newcommand{\Rmnum}[1]{\expandafter\@slowromancap\romannumeral #1@}
\makeatother

\begin{document}
\renewcommand{\theequation}{\arabic{equation}}
\numberwithin{equation}{section}

\title[Special Unipotent Orbits and Fourier Coefficients]{On Special Unipotent Orbits and Fourier Coefficients for Automorphic Forms on Symplectic Groups}

\author{Dihua Jiang}
\address{School of Mathematics\\
University of Minnesota\\
Minneapolis, MN 55455, USA}
\email{dhjiang@math.umn.edu}

\author{Baiying Liu}
\address{Department of Mathematics\\
University of Utah\\
Salt Lake City, UT 84112, USA}
\email{liuxx969@umn.edu}


\date{\today}

\dedicatory{In memory of Steve Rallis}

\thanks{The work of the first named author is supported in part by the NSF Grant DMS--1001672 and DMS--1301567}

\begin{abstract}
Fourier coefficients of automorphic representations $\pi$ of $\Sp_{2n}(\BA)$ are attached to unipotent adjoint orbits in $\Sp_{2n}(F)$, where
$F$ is a number field and $\BA$ is the ring of adeles of $F$. We prove that for a given $\pi$, all maximal unipotent orbits that gives nonzero
Fourier coefficients of $\pi$ are special, and prove, under a well acceptable assumption, that if $\pi$ is cuspidal, then
the stabilizer attached to each of those maximal unipotent orbits is $F$-anisotropic as algebraic group over $F$. These results
strengthen, refine and extend the earlier work of Ginzburg, Rallis and Soudry on the subject. As a consequence, we obtain constraints on
those maximal unipotent orbits if $F$ is totally imaginary, further applications of which to the discrete spectrum with the
Arthur classification will be considered in our future work.
\end{abstract}

\maketitle


\section{Introduction}

Fourier coefficients of automorphic representations of a reductive algebraic group $G$ over a number field $F$ have been used by Steve
Rallis in most of his mathematics work. Motivated by his work on constructions of integral representations for various automorphic $L$-functions,
and on constructions of automorphic descents for the Langlands functorial relations between classical groups and the general linear groups,
Rallis with Ginzburg and Soudry started in their 2003 paper (\cite{GRS03}) a systematic investigation of Fourier coefficients of automorphic
representations of the symplectic group $\Sp_{2n}(\BA)$, where $\BA$ is the ring of adeles of $F$, attached to unipotent adjoint orbits
in $\Sp_{2n}(F)$, and obtained some basic structures of the Fourier coefficients.

In this paper, we follow \cite{GRS03} and also the work of Ginzburg (\cite{G06}) to obtain more refined structures and properties of
Fourier coefficients of automorphic representations of $\Sp_{2n}(\BA)$. The results obtained here can be viewed as natural
extensions of their earlier work. Following the general framework of \cite{J14}, they are expected to have more impacts to the
discrete spectrum of square-integrable automoprhic forms with the Arthur classification (\cite{Ar13} and \cite{Mk14}).

We recall some basics of Fourier coefficients of automorphic representations in Section 2. In order to strengthen and extend the results
from \cite{GRS03} and \cite{G06}, we prove two properties of Fourier coefficients attached to composite partitions (Propositions 3.2 and 3.3),
including a complete proof of a technical lemma (Lemma 3.1), which is Lemma 2.6 of \cite{GRS03}. In Section 4, we prove one of the main results
of this paper (Theorem 4.1) that for any irreducible automorphic representation $\pi$ of $\Sp_{2n}(\BA)$, if $\pi$ has
a nonzero Fourier coefficient attached to a symplectic partition $\udl{p}$, then it has also a nonzero Fourier coefficient attached to
the symplectic extension $\udl{p}^{\Sp_{2n}}$ of $\udl{p}$. This implies immediately that if $\udl{p}$ is a maximal partition
to which a nonzero Fourier coefficient of $\pi$ attached, then $\udl{p}$ is special. This result is motivated by the corresponding
local result by M{\oe}glin in \cite{M96} and was proved in \cite{GRS03}. We provide here an algorithm which sees an inductive way to
give a complete proof. In the last section (Section 5), we prove our second main result (Theorem 5.2) that
for a maximal partition $\udl{p}$ to which a nonzero Fourier coefficient of an irreducible cuspidal automorphic representation $\pi$ is attached,
the stabilizer in a Levi subgroup $L_{\udl{p}}$ of the corresponding character $\psi_{\udl{p}.\udl{a}}$ is $F$-anisotropic. The notation
will be defined Section 2. In fact, we prove this result under an assumption that for an irreducible cuspidal automorphic representation $\pi$,
there exists a unique maximal partition to which a nonzero Fourier coefficient of $\pi$ is attached. We are going to investigate this
assumption in our future work. We also refer to \cite{G06} for a more general conjecture on this issue. When the group is $\GL_n$,
this uniqueness was proved by Piatetski-Shapiro in \cite{PS79} and by Shalika in \cite{Sl74} for $\pi$ to be cuspidal, and
by the authors in \cite{JL13} for $\pi$ to be in the discrete spectrum of $\GL_n$. One of the interesting consequences of
Theorem 5.2 is to provide constraints on the maximal partitions $\udl{p}$ to which a nonzero Fourier coefficient of an irreducible cuspidal
automorphic representation $\pi$ is attached (Theorem 5.3). Further application of this result to the understanding of discrete spectrum with
the Arthur classification will be considered in our future work within the framework of \cite{J14}.

We would like to thank David Vogan for stimulating conversation on the topic, which brings our attention to the statement of Theorem 5.2, and
for bringing our attention to his student's thesis \cite{H11}. We would like to thank Joseph Hundley and Monica Nevins for helpful discussions on relevant topics.
We would also like to thank the referee for the useful suggestions and comments on an early version of this paper.


\section{Fourier Coefficients For Symplectic Groups}

We recall the definition of Fourier coefficients
attached to unipotent orbits of symplectic groups. Let $G_n=\Sp_{2n}$ be the symplectic group with a Borel subgroup $B=TU$, where the
maximal torus $T$ consists of all diagonal matrices of form:
$
\diag(t_1,\cdots,t_n;t_n^{-1},\cdots,t_1^{-1})
$
and the unipotent radical of $B$ consists all upper unipotent matrices in $\Sp_{2n}$.

Let $F$ be a number field and $\ovl{F}$ be the algebraic closure of $F$. It is well known that
the set of all unipotent adjoint orbits of $G_n(\ovl{F})$ is parameterized by
the set of partitions of $2n$ whose odd parts occur with even multiplicity (see \cite{CM93}, \cite{N11} and \cite{W01}, for instance).
We may call them symplectic partitions of $2n$.

If we consider $G_n=\Sp_{2n}$ over $F$, the symplectic partitions of $2n$ parameterize the $F$-stable unipotent orbits of $G_n(F)$.
In order to define the Fourier coefficients of automorphic forms of $G_n(\BA)$ attached to an $F$-unipotent orbit,
where $\BA$ is the ring of adeles of $F$, we introduce certain explicit elements in the $F$-unipotent orbit.
We do this following \cite{CM93}, and see also
\cite{N11} and \cite{W01}.

Given a symplectic partition $\underline{p}=[p_1, \ldots, p_t]$,
the Recipe 5.2.2 for Type $C_n$ (\cite[Page 77]{CM93})
gives a way to attach to $\underline{p}$ a standard $\frak{sl}_2$-triple. It is easy to see that this method also works over $F$ by using
\cite{N11} and \cite{W01}. One may break the series $\{p_1, \ldots, p_t\}$ into chunks of the
following two types: $\{2k+1, 2k+1\}$ or $\{2k\}$, and rewrite the partition as $\underline{p}=[p_1^{e_1} \cdots p_r^{e_r}]$,
with $p_1 \geq p_2 \geq \cdots \geq p_r$, where $e_i = 1$ if $p_i$ is even,
$e_i =2$, if $p_i$ is odd. From now on, general symplectic partitions will always be written
in this form.

We assign a block of consecutive indices
for each $p_i$ inductively. For $p_1$, choose the block $\{1, \ldots, q_1\}$,
where $q_1 = \frac{p_1}{2}$ if $e_1=1$, and $q_1 = p_1$
if $e_1=2$. Assume that for $p_{i-1}$, $i \geq 2$, the block of consecutive indices
has been chosen to be $\{\sum_{j=1}^{i-2} q_j +1, \ldots, \sum_{j=1}^{i-1} q_j\}$. Then
the block for $p_i$ is $\{\sum_{j=1}^{i-1} q_j +1, \ldots, \sum_{j=1}^{i} q_j\}$,
where $q_i = \frac{p_i}{2}$ if $e_i=1$, and $q_i = p_i$ if $e_i=2$.

Next, to each $p_i$, we assign a set of negative roots $X_i$ as follows:
for $1 \leq i \leq r$, define
$$
X_i = \{\alpha^i_1 = e_{N+2}-e_{N+1},
\ldots, \alpha^i_{q_i-1} = e_{N+q_i}-e_{N+q_i-1}, \alpha^i_{q_i} = -2 e_{N+q_i}\}
$$
if $e_i = 1$, and define
$$
X_i = \{\alpha^i_1 = e_{N+2}-e_{N+1},
\ldots, \alpha^i_{q_i-1} = e_{N+q_i}-e_{N+q_i-1}\}
$$
if $e_i = 2$,
where $N = \sum_{j=1}^{i-1} q_j$. Then we define
$$
X_{\underline{p}, \underline{a}} = \sum_{i=1}^r \sum_{j=1}^{q_i-1} x_{\alpha^i_j}(\frac{1}{2}) + \sum_{e_i = 1} x_{\alpha^i_{q_i}} (a_i),
$$
where $a_i \in F^*/(F^*)^2$ if $e_i=1$, which defines
$\underline{a} = \{a_i : e_i = 1\}$, and for each root $\alpha$, $x_{\alpha}(x)$ is defined to be the a root vector in the
corresponding root subspace.
Then, we can find $Y_{\underline{p}, \underline{a}}$ a sum of positive root vectors and a semisimple element $H_{\underline{p}, \underline{a}}$,
such that $\{H_{\underline{p}, \underline{a}}, Y_{\underline{p}, \underline{a}}, X_{\underline{p}, \underline{a}}\}$ is a
standard $\mathfrak{sl}_2$-triple. Since we will only use $X_{\underline{p}, \underline{a}}$ to
define the Fourier coefficients attached to the partition $\underline{p}=[p_1^{e_1} \cdots p_r^{e_r}]$,
we will not make $Y_{\underline{p}, \underline{a}}$ and $H_{\underline{p}, \underline{a}}$ more explicit here.

To this $\mathfrak{sl}_2$-triple $\{H_{\underline{p}, \underline{a}}, Y_{\underline{p}, \underline{a}}, X_{\underline{p}, \underline{a}}\}$,
we define a one-dimensional toric subgroup $\mathcal{H}_{\underline{p}}$
of $G_n(F)$: for $t \in F^*$,
\begin{equation}\label{equtorus}
\mathcal{H}_{\underline{p}}(t):= \diag(T_1(t), \ldots T_r(t), T_r^*(t), \ldots, T_1^*(t)),
\end{equation}
where
$$
T_i(t) =
\begin{cases}
 \diag(t^{q_i-1}, \ldots, t)& \text{if}\  e_i =1;\\
\diag(t^{q_i-1}, \ldots, t^{1-q_i})& \text{if}\ e_i =2.
\end{cases}
$$
It is easy to see that under the adjoint action,
$$
\Ad(\mathcal{H}_{\underline{p}}(t))(X_{\underline{p}, \underline{a}}) = t^{-2} X_{\underline{p}, \underline{a}}, \forall t \in F^*.
$$

Let $\mathfrak{g}$ be the Lie algebra of $G_n$. Under the
adjoint action of $\mathcal{H}_{\underline{p}}$,
$\mathfrak{g}$ has the following direct sum decomposition into
$\mathcal{H}_{\underline{p}}$-eigenspaces:
\begin{equation}\label{ssd}
\mathfrak{g} = \mathfrak{g}_{-m} \oplus \cdots \oplus \mathfrak{g}_{-2} \oplus \mathfrak{g}_{-1} \oplus \mathfrak{g}_{0} \oplus \mathfrak{g}_{1} \oplus \mathfrak{g}_{2} \oplus \cdots \oplus \mathfrak{g}_{m},
\end{equation}
for some positive integer $m$, where
$\mathfrak{g}_{l}:= \{X \in \mathfrak{g}|\Ad(\mathcal{H}_{\underline{p}}(t))(X) = t^{l}X\}$.

Let $V_{\underline{p}, j}$ ($j=1, \ldots, m$) be the unipotent
subgroup of $G_n(F)$ with Lie algebra $\oplus_{l=j}^m \mathfrak{g}_l$.
Let $L_{\underline{p}}$ be the algebraic subgroup of $G_n(F)$
with Lie algebra $\mathfrak{g}_{0}$. Under the adjoint action,
the set $\Ad(L_{\underline{p}}(\overline{F}))(X_{\underline{p}, \underline{a}})$
is Zariski open dense in dual space of $V_{\underline{p}, 2}(\overline{F})/V_{\underline{p}, 3}(\overline{F})$.
Note that when passing to $F$ from $\overline{F}$, the $F$-points of this Zariski open dense orbit
decomposes as a union of $F$-rational orbits, which form the corresponding $F$-stable orbit defined over $F$.
By Proposition 5 of \cite{N11} (see also \cite{W01}), when
the $a_i$'s in $\ul{a}$ run through all the square classes in $F^*/(F^*)^2$,
$\Ad(L_{\underline{p}}(F))(X_{\underline{p}, \underline{a}})$
gives all the $F$-rational orbits in the corresponding $F$-stable orbit defined over $F$.

Each $F$-rational $\Ad(L_{\underline{p}}(F))(X_{\underline{p}, \underline{a}})$ defines a character of $V_{\underline{p}, 2}$
as follows: for $v \in V_{\underline{p}, 2}(\BA)$ and for a nontrivial character $\psi$ of $F\bks\BA$,
\begin{align}\label{ch}
\begin{split}
\psi_{\underline{p}, \underline{a}}(v) := & \psi(\tr(X_{\underline{p}, \underline{a}}\log(v)))\\
= \ & \prod_{i=1}^r \psi(v_{\sum_{j=1}^{i-1} q_j +1, \sum_{j=1}^{i-1} q_j +2} +
\cdots + v_{\sum_{j=1}^{i} q_j -1, \sum_{j=1}^{i} q_j})\\
& \cdot \psi(\sum_{e_i =1} a_i v_{\sum_{j=1}^{i} q_j, \sum_{j=1}^{i} q_j}).
\end{split}
\end{align}

From now on, we may take $G_n(\BA)=\Sp_{2n}(\BA)$ or $\tilsp_{2n}(\BA)$, unless specified otherwise.
Since $G_n(\BA) = \tilsp_{2n}(\BA)$
splits over any unipotent subgroup, the above discussion holds for
$G_n(\BA) = \tilsp_{2n}(\BA)$.

For an arbitrary automorphic form $\varphi$ on $G_n(\BA)$, the $\psi_{\underline{p}, \underline{a}}$-Fourier coefficient
of $\varphi$ is defined by
\begin{equation}\label{fc}
\varphi^{\psi_{\underline{p}, \underline{a}}}(g):=\int_{[V_{\underline{p}, 2}]}
\varphi(vg) \psi^{-1}_{\underline{p}, \underline{a}}(v) dv,
\end{equation}
where for a group $G$, we denote the quotient $G(F) \bs G(\BA)$ simply by $[G]$.
Since $\varphi$ is automorphic, the non-vanishing
of the $\psi_{\underline{p}, \underline{a}}$-Fourier coefficient
of $\varphi$ depends only on the $F$-rational orbit $\Ad(L_{\underline{p}}(F))(X_{\underline{p}, \underline{a}})$.
When an irreducible automorphic representation $\pi$ of $G_n(\BA)$ is generated by automorphic forms $\varphi$,
we say that $\pi$ has a nonzero $\psi_{\underline{p}, \underline{a}}$-Fourier coefficient or
a nonzero Fourier coefficient attached to $\udl{p}$ if there exists an
automorphic form $\varphi$ in the space of $\pi$ with a nonzero $\psi_{\underline{p}, \underline{a}}$-Fourier coefficient
$\varphi^{\psi_{\underline{p}, \underline{a}}}(g)$, for some choice of $\ul{a}$.

For any irreducible automorphic representation $\pi$
of $G_n(\BA)$, as in \cite{J14}, we define $\mathfrak{n}^m(\pi)$ to be the set of
all symplectic partitions $\underline{p}$ which have the properties
that $\pi$ has a nonzero $\psi_{\underline{p}, \underline{a}}$-Fourier coefficient
for some choice of $\underline{a}$, and for any ${\underline{p}'} > \underline{p}$ (with
the natural ordering of partitions), $\pi$ has no nonzero Fourier coefficients
attached to ${\underline{p}'}$.


Recall that $V_{\udl{p},1}(F)$ is the unipotent subgroup
of $G_n(F)$ whose Lie algebra is $\mathfrak{g}_1\oplus\mathfrak{g}_2\oplus\cdots\oplus\mathfrak{g}_m$.
Following \cite{MW87}, we define
$$
(X_{\underline{p}, \underline{a}})^{\sharp}= \{X \in \mathfrak{g}\ |\
\tr(X_{\underline{p}, \underline{a}}[X, X']) =0, \forall X' \in \mathfrak{g} \}.
$$
Define $V_{\udl{p},2}' = \exp(\mathfrak{g}_1 \cap (X_{\underline{p}, \underline{a}})^{\sharp}) V_{\udl{p},2}$,
which is a normal subgroup of $V_{\udl{p},1}(F)$.
From the definition of $(X_{\underline{p}, \underline{a}})^{\sharp}$, it is easy to see that
the character $\psi_{\underline{p}, \underline{a}}$ on $V_{\udl{p},2}$ can be trivially
extended to $V_{\udl{p},2}'$, which we still denote by $\psi_{\underline{p}, \underline{a}}$.
It turns out that $V_{\udl{p},1} / \ker_{V_{\udl{p},2}'}(\psi_{\underline{p}, \underline{a}})$ has a Heisenberg
structure $W \oplus Z$ (see \cite{MW87}, Section \Rmnum{1}.7), where
$W \cong V_{\udl{p},1} / V_{\udl{p},2}'$, and
$Z \cong V_{\udl{p},2}'/\ker_{V_{\udl{p},2}'}(\psi_{\underline{p}, \underline{a}})$.
Note that the symplectic form on $W$ is the one
inherited from the Lie algebra bracket, i.e.,
for $w_1, w_2 \in W$ (here, we identify $w \in W$ with any of it's representatives in $V_{\udl{p},1}$
such that $\log(w) \in \mathfrak{g}_1$),
$
\langle w_1, w_2 \rangle  = \tr(X_{\underline{p}, \underline{a}}\log([w_1, w_2]))
= \tr(X_{\underline{p}, \underline{a}}[\log(w_1), \log(w_2)]).
$
The non-degeneracy of this symplectic form can be checked easily
as following: for fixed $w_1 \in W$, if $\langle w_1, w_2 \rangle \equiv 0$ for any $w_2 \in W$, that is, $$\tr(X_{\underline{p}, \underline{a}}[\log(w_1), \log(w_2)]) \equiv 0$$
 for any $w_2 \in W$, then,
$
\tr(X_{\underline{p}, \underline{a}}[\log(w_1), X']) \equiv 0,
$
for any $X' \in \mathfrak{g}_1$. Therefore,
$\tr(X_{\underline{p}, \underline{a}}[\log(w_1), X']) \equiv 0$,
for any $X' \in \mathfrak{g}$. Hence, $\log(w_1)\in (X_{\underline{p}, \underline{a}})^{\sharp}$,
that is, $w_1 = 0 \in V_{\udl{p},1} / V_{\udl{p},2}'$.

As in the case of $\GL_n$ (see Lemma 5.1, \cite{JL13}), we can prove that
actually $V_{\udl{p},2}' = V_{\udl{p},2}$, i.e., $\mathfrak{g}_1 \cap (X_{\underline{p}, \underline{a}})^{\sharp} = \{0\}$. 

\begin{lem}\label{key1}
The two subgroups $V_{\udl{p},2}'$ and $V_{\udl{p},2}$ are equal. In particular,
the quotient group $V_{\udl{p},1} / \ker_{V_{\udl{p},2}}(\psi_{\underline{p}, \underline{a}})$ has a Heisenberg
structure $W \oplus Z$, where
$W \cong V_{\udl{p},1} / V_{\udl{p},2}$ is symplectic, and
$Z \cong V_{\udl{p},2}/\ker_{V_{\udl{p},2}}(\psi_{\underline{p}, \underline{a}})$ is the center.
\end{lem}

\begin{proof}
If $\mathfrak{g}_1$ is zero, there is nothing to prove. We assume that $\mathfrak{g}_1 \neq \{0\}$. To prove $V_{\udl{p},2}' = V_{\udl{p},2}$,
it suffices to prove that $\mathfrak{g}_1 \cap (X_{\underline{p}, \underline{a}})^{\sharp} = \{0\}$.

First, we are going to describe structure of elements in $V_{\udl{p},2}$.
By definition, elements $v$ in $V_{\udl{p},2}$ can be written as a product of
\begin{align*}
\begin{pmatrix}
\ddots & & & & & & & & \\
& n_i &  & q^{i,j}_1 &  & q^{i,j}_2 & & q^{i,i} & \\
&  & \ddots & & & & & & \\
& 0 & & n_j & & q^{j,j} & & (q^{i,j}_2)^* & \\
& & & &\ddots & & & & \\
& 0 & & 0 & & (n_j)^* & & (q^{i,j}_1)^* & \\
& & & & & & \ddots & & \\
& 0 & & 0 & & 0 & & (n_i)^* & \\
& & & & & & & & \ddots
\end{pmatrix}
\end{align*}
and
\begin{align*}
\begin{pmatrix}
\ddots & & & & & & & & \\
& I_{q_i} &  & 0 &  & 0 & & 0 & \\
&  & \ddots & & & & & & \\
& p^{i,j}_1 & & I_{q_j} & & 0 & & 0 & \\
& & & &\ddots & & & & \\
& p^{i,j}_2 & & p^{j,j} & & I_{q_j} & & 0 & \\
& & & & & & \ddots & & \\
& p^{i,i} & & (p^{i,j}_2)^* & & (p^{i,j}_1)^* & & I_{q_i} & \\
& & & & & & & & \ddots
\end{pmatrix},
\end{align*}
where $n_i \in N_{q_i}$ ($n_j \in N_{q_j}$), the maximal upper-triangular unipotent subgroup
of $\GL_{q_i}$ ($\GL_{q_j}$, respectively),
$q^{i,j}_1, q^{i,j}_2 \in M_{q_i \times q_j}$, $q^{i,i} \in M_{q_i \times q_i}$,
$(q^{i,j}_1)^*, (q^{i,j}_2)^* \in M_{q_j \times q_i}$, $q^{j,j} \in M_{q_j \times q_j}$,
$p^{i,j}_1, p^{i,j}_2 \in M_{q_j \times q_i}$, $p^{i,i} \in M_{q_i \times q_i}$,
$(p^{i,j}_1)^*, (p^{i,j}_2)^* \in M_{q_i \times q_j}$, and $p^{j,j} \in M_{q_j \times q_j}$,
satisfying certain conditions. By definition, $\mathfrak{g}_1\neq 0$ if there are parts $p_i$ and $p_j$ having different parity.
The above ``certain conditions" are needed for the case that the parts $p_i$
and $p_j$ are of different parity. We only need to describe
them for $q^{i,j}_1, q^{i,j}_2$ and $p^{i,j}_1, p^{i,j}_2$. There are two sub-cases to be considered:
(1) $p_i$ is even and $p_j$ is odd and (2) $p_i$ is odd and $p_j$ is even, assuming that $i<j$.

For Case (1) where $p_i$ is even and $p_j$ is odd, we have that $q_i = \frac{p_i}{2}$ and $q_j = p_j$. Then
$q^{i,j}_k \in M_{q_i \times q_j}$ with $k=1,2$ and $q^{i,j}_k(l,m)=0$ for $l \geq m + \frac{p_i-p_j-1}{2}$;
and $p^{i,j}_k \in M_{q_j \times q_i}$ with $k=1,2$ and $p^{i,j}_k(l,m)=0$
for $m \leq l + \frac{p_i-p_j-1}{2}+1$.

For Case (2) where $p_i$ is odd and $p_j$ is even, we have that $q_i = p_i$ and $q_j = \frac{p_j}{2}$. Then
$q^{i,j}_1 \in M_{q_i \times q_j}$ with $q^{i,j}_1(l,m)=0$
for $l \geq m + \frac{p_i-p_j-1}{2}$; and $q^{i,j}_2 \in M_{q_i \times q_j}$ with $q^{i,j}_2(l,m)=0$
for $l \geq m + \frac{p_j}{2} + \frac{p_i-p_j-1}{2}$. Also $p^{i,j}_1 \in M_{q_j \times q_i}$ with $p^{i,j}_1(l,m)=0$
for $m \leq l + \frac{p_i-p_j-1}{2}+1$; and
$p^{i,j}_2 \in M_{q_j \times q_i}$ with $p^{i,j}_2(l,m)=0$
for $m \leq l + \frac{p_j}{2} + \frac{p_i-p_j-1}{2}+1$.

For Case (1), we define abelian groups $Y^{i,j}_1$
and $X_1^{i,j}$ as follows. First we define
$$
Y^{i,j}_1 = \prod_{l=1}^{\frac{p_j+1}{2}} X_{\alpha^{i,j}_l}(y^{i,j}_l)
\prod_{m=1}^{\frac{p_j-1}{2}} X_{\beta^{i,j}_m}(y^{i,j}_m),
$$
where $\alpha^{i,j}_l = e_{\sum_{k=1}^{i-1}q_k + \frac{p_i-p_j-1}{2} +l}
+ e_{n - \sum_{k=j+1}^{r}q_k -l+1}$ for $1 \leq l \leq \frac{p_j+1}{2}$, and
$\beta^{i,j}_m = e_{\sum_{k=1}^{i-1}q_k + \frac{p_i-p_j-1}{2} +m}
- e_{\sum_{k=1}^{j-1}q_k +m}$
for $1 \leq m \leq \frac{p_j-1}{2}$. Here to each given root $\alpha$, $X_{\alpha}(x)$ is the corresponding one-dimensional root subgroup.
Then we define
$$
X^{i,j}_1 = \prod_{l=1}^{\frac{p_j+1}{2}} X_{\gamma^{i,j}_l}(x^{i,j}_l)
\prod_{m=1}^{\frac{p_j-1}{2}} X_{\delta^{i,j}_m}(x^{i,j}_m),
$$
where $\gamma^{i,j}_l = e_{\sum_{k=1}^{j-1}q_k +l}
- e_{\sum_{k=1}^{i-1}q_k +\frac{p_i-p_j-1}{2}+l+1}$ for $1 \leq l \leq \frac{p_j-1}{2}$;
$\gamma^{i,j}_{\frac{p_j+1}{2}} = e_{\sum_{k=1}^{i-1}q_k + \frac{p_i-p_j-1}{2} +\frac{p_j+1}{2}}
- e_{\sum_{k=1}^{j-1}q_k +\frac{p_j+1}{2}}$;
and $\delta^{i,j}_m = - e_{n - \sum_{k=j+1}^{r}q_k -m+1}
- e_{\sum_{k=1}^{i-1}q_k +\frac{p_i-p_j-1}{2}+m+1}$ for $1 \leq m \leq \frac{p_j-1}{2}$.

For Case (2),  we define abelian groups $Y^{i,j}_2$
and $X_2^{i,j}$ as follows. First, we define
$$
Y^{i,j}_2 = \prod_{l=1}^{\frac{p_j}{2}} X_{\alpha^{i,j}_l}(y^{i,j}_l)
\prod_{m=1}^{\frac{p_j}{2}} X_{\beta^{i,j}_m}(y^{i,j}_m),
$$
where $\alpha^{i,j}_l = e_{\sum_{k=1}^{i-1}q_k + \frac{p_i-p_j-1}{2} +l}
- e_{\sum_{k=1}^{j-1}q_k +l}$ for $1 \leq l \leq \frac{p_j}{2}$, and
$\beta^{i,j}_m = e_{\sum_{k=1}^{i-1}q_k + \frac{p_i-p_j-1}{2} +\frac{p_j}{2}+m}
+ e_{n-\sum_{k=j+1}^{r}q_k-m+1}$ for $1 \leq m \leq \frac{p_j}{2}$. Then we define
$$
X^{i,j}_2 = \prod_{l=1}^{\frac{p_j}{2}} X_{\gamma^{i,j}_l}(y^{i,j}_l)
\prod_{m=1}^{\frac{p_j}{2}} X_{\delta^{i,j}_m}(y^{i,j}_m),
$$
where $\gamma^{i,j}_l = e_{\sum_{k=1}^{j-1}q_k + l}
- e_{\sum_{k=1}^{i-1}q_k +\frac{p_i-p_j-1}{2} +l+1}$ for $1 \leq l \leq \frac{p_j}{2}$, and
$\delta^{i,j}_m = -e_{n -\sum_{k=j+1}^{r}q_k -m+1}
- e_{\sum_{k=1}^{i-1}q_k+ \frac{p_i-p_j-1}{2} +\frac{p_j}{2}+m+1}$ for $1 \leq m \leq \frac{p_j}{2}$.

Define $Y^{i,j} = Y^{i,j}_k$ and $X^{i,j} = X^{i,j}_k$ for $k=1,2$, depending on the case of the pair $(p_i, p_j)$.
Then define
\begin{align}\label{equy}
Y = & \prod_{1 \leq i < j \leq r \text{, } p_i \text{ and } p_j \text{ are of different parity }} Y^{i,j},
\end{align}
and
\begin{align}\label{equx}
X = & \prod_{1 \leq i < j \leq r \text{, } p_i \text{ and } p_j \text{ are of different parity }} X^{i,j}.
\end{align}
It follows that $\mathfrak{g}_1 = \log(X) \oplus \log(Y)$. To show
that $\mathfrak{g}_1 \cap (X_{\underline{p}, \underline{a}})^{\sharp} = \{0\}$, it is enough to show that
$$
(\log(X) \oplus \log(Y)) \cap (X_{\underline{p}, \underline{a}})^{\sharp} = \{0\}.
$$
In other words, it suffices to show that
for $1 \leq i < j \leq r$ such that $p_i$ and $p_j$ are of different parity,
$$
(\log(X^{i,j}) \oplus \log(Y^{i,j})) \cap (X_{\underline{p}, \underline{a}})^{\sharp} = \{0\}.
$$

If $(p_i, p_j)$ is in Case (1), i.e., $p_i$ is even, $p_j$ is odd,
then by direct calculation, we obtain that when $1 \leq l \leq \frac{p_j-1}{2}$,
\begin{align*}
\tr(X_{\underline{p}, \underline{a}}[\log(X_{\delta^{i,j}_l}(x^{i,j}_l)), \log(X_{\alpha^{i,j}_l}(y^{i,j}_l))])
& = -x^{i,j}_l y^{i,j}_l;\\
\tr(X_{\underline{p}, \underline{a}}[\log(X_{\gamma^{i,j}_l}(x^{i,j}_l)), \log(X_{\beta^{i,j}_l}(y^{i,j}_l))])
& = -x^{i,j}_l y^{i,j}_l.
\end{align*}
and when $l = \frac{p_j+1}{2}$,
$$
\tr(X_{\underline{p}, \underline{a}}[\log(X_{\alpha^{i,j}_l}(y^{i,j}_l)), \log(X_{\gamma^{i,j}_l}(x^{i,j}_l))])
= -a_i x^{i,j}_l y^{i,j}_l.
$$
This implies that
$(\log(X^{i,j}_1) \oplus \log(Y^{i,j}_1)) \cap (X_{\underline{p}, \underline{a}})^{\sharp} = \{0\}$.

If $(p_i, p_j)$ is in Case (2), i.e., $p_i$ is odd, $p_j$ is even,
then by direct calculation, we get that for $1 \leq l \leq \frac{p_j}{2}$,
\begin{align*}
\tr(X_{\underline{p}, \underline{a}}[\log(X_{\gamma^{i,j}_l}(x^{i,j}_l)), \log(X_{\alpha^{i,j}_l}(y^{i,j}_l))])
& = -x^{i,j}_l y^{i,j}_l;\\
\tr(X_{\underline{p}, \underline{a}}[\log(X_{\delta^{i,j}_l}(x^{i,j}_l)), \log(X_{\beta^{i,j}_l}(y^{i,j}_l))])
& = -x^{i,j}_l y^{i,j}_l.
\end{align*}
Therefore, we also have that
$(\log(X^{i,j}_2) \oplus \log(Y^{i,j}_2)) \cap (X_{\underline{p}, \underline{a}})^{\sharp} = \{0\}$.
This completes the proof of this lemma.
\end{proof}

From the proof of Lemma \ref{key1}, we have

\begin{cor}\label{polar}
The subspaces $X$ and $Y$ as defined in \eqref{equx}, \eqref{equy} give a polarization $X \oplus Y$ of $W$.
\end{cor}

\begin{exmp}
\begin{enumerate}
\item Let $n=3$, then Fourier coefficients of automorphic forms $\varphi$ of $\Sp_6(\BA)$ attached to the partition $[41^2]$ with respect to a square class $\alpha$ are defined as follows.
$\CH_{[41^2]}(t)=\diag(t^3, t, 1, 1, t^{-1}, t^{-3})$, $\mathfrak{g} =  \mathfrak{g}_{-6} \oplus \mathfrak{g}_{-4} \oplus \mathfrak{g}_{-2} \oplus \mathfrak{g}_{-1} \oplus \mathfrak{g}_{0} \oplus \mathfrak{g}_{1} \oplus \mathfrak{g}_{2} \oplus \mathfrak{g}_{4} \oplus \mathfrak{g}_{6},$
   and $V_{[41^2], 2}=\exp(\mathfrak{g}_{2} \oplus \mathfrak{g}_{4} \oplus \mathfrak{g}_{6}).$ Elements in $V_{[41^2], 2}$ have the following form:
    $$\begin{pmatrix}
	1 & * & * & * & * & *\\
	0 & 1 & 0 & 0 & * & *\\
	0 & 0 & 1 & 0 & 0 & *\\
	0 & 0 & 0 & 1 & 0 & *\\
	0 & 0 & 0 & 0 & 1 & *\\
	0 & 0 & 0 & 0 & 0 & 1
	\end{pmatrix}.$$
	For any $v \in V_{[41^2], 2}$, $\psi_{[41^2], \alpha}(v)=\psi(v_{1,2} + \alpha v_{2,5})$. Then the corresponding Fourier coefficient of $\varphi$ is
	\begin{equation*}
\varphi^{\psi_{[41^2], \alpha}}(g)=\int_{[V_{[41^2], 2}]}
\varphi(vg) \psi^{-1}_{[41^2], \alpha}(v) dv.
\end{equation*}
Define roots $\alpha_1 = e_2 +e_3$, $\alpha_2 = e_2 -e_3$. Let $X_{\alpha_i}$ be the one-dimensional root subgroup corresponding to $\alpha_i$, $i=1,2$. Then the unipotent subgroup of $Y$ defined in \eqref{equy} equals $X_{\alpha_1}$, the unipotent subgroup of $X$ defined in \eqref{equx} equals $X_{\alpha_2}$.
\item Let $n=5$, then Fourier coefficients of automorphic forms $\varphi$ of $\Sp_{10}(\BA)$ attached to the partition $[43^2]$ with respect to a square class $\alpha$ are defined as follows.
$$\CH_{[43^2]}(t)=\diag(t^3, t, t^2, 1, t^{-2}, t^2, 1, t^{-2}, t^{-1}, t^{-3}),$$
 $\mathfrak{g} =  \mathfrak{g}_{-6} \oplus \mathfrak{g}_{-5} \oplus \mathfrak{g}_{-4} \oplus \mathfrak{g}_{-3} \oplus \mathfrak{g}_{-2} \oplus \mathfrak{g}_{-1} \oplus \mathfrak{g}_{0} \oplus \mathfrak{g}_{1} \oplus \mathfrak{g}_{2} \oplus \mathfrak{g}_{3} \oplus \mathfrak{g}_{4} \oplus \mathfrak{g}_{5} \oplus \mathfrak{g}_{6},$
   and $V_{[43^2], 2}=\exp(\mathfrak{g}_{2} \oplus \mathfrak{g}_{3} \oplus \mathfrak{g}_{4} \oplus \mathfrak{g}_{5} \oplus \mathfrak{g}_{6}).$ Elements in $V_{[43^2], 2}$ have the following form:
    $$\begin{pmatrix}
	1 & * & 0 & * & * & 0 & * & * &*&*\\
	0 & 1 & 0 & 0 & * & 0& 0 & * &*&*\\
	0 & 0 & 1 & * & * & 0& * & * &*&*\\
	0 & 0 & 0 & 1 & * & 0& 0 & * &0&*\\
	0 & 0 & 0 & 0 & 1 & 0& 0 & 0 &0&0\\
	0 & 0 & 0 & * & * & 1& * & * &*&*\\
    0 & 0 & 0 & 0 & * & 0& 1 & * &0&*\\
	0 & 0 & 0 & 0 & 0 & 0& 0 & 1 &0&0\\
	0 & 0 & 0 & 0 & 0 & 0& 0 & 0 &1&*\\
	0 & 0 & 0 & 0 & 0 & 0& 0 & 0 &0&1
	\end{pmatrix}.$$
	For any $v \in V_{[43^2], 2}$, $\psi_{[43^2], \alpha}(v)=\psi(v_{1,2} + \alpha v_{2,9} + v_{3,4}+v_{4,5})$. Then the corresponding Fourier coefficient of $\varphi$ is
	\begin{equation*}
\varphi^{\psi_{[43^2], \alpha}}(g)=\int_{[V_{[43^2], 2}]}
\varphi(vg) \psi^{-1}_{[43^2], \alpha}(v) dv.
\end{equation*}
Define roots $\alpha_1 = e_1-e_3$, $\alpha_2=e_1+e_5$, $\alpha_3=e_2+e_4$, $\alpha_4 = e_3-e_2$, $\alpha_5=-e_5-e_2$, $\alpha_6=e_2-e_4$. Let $X_{\alpha_i}$ be the one-dimensional root subgroup corresponding to $\alpha_i$, $i=1,\cdots,6$. Then the unipotent subgroup of $Y$ defined in \eqref{equy} equals $\prod_{i=1}^3 X_{\alpha_i}$, the unipotent subgroup of $X$ defined in \eqref{equx} equals $\prod_{i=4}^6 X_{\alpha_i}$.
\end{enumerate}
\end{exmp}

In general explicit calculations of Fourier coefficients of automorphic forms when $\mathfrak{g}_1$ is nonzero,
there is a very useful lemma, which has been formulated in a general term as Corollary 7.1 of \cite{GRS11}
and as Lemma 5.2 of \cite{JL13}
in a slightly different way. We recall this lemma below and apply it to $G_n$.

Let $C$ be an $F$-subgroup of a maximal unipotent subgroup of $G_n$, and let $\psi_C$ be a non-trivial character of $[C] = C(F) \bs C(\BA)$.
$\wt{X}, \wt{Y}$ are two unipotent $F$-subgroups, satisfying the following conditions:
\begin{itemize}
\item[(1)] $\wt{X}$ and $\wt{Y}$ normalize $C$;
\item[(2)] $\wt{X} \cap C$ and $\wt{Y} \cap C$ are normal in $\wt{X}$ and $\wt{Y}$, respectively, $(\wt{X} \cap C) \bs \wt{X}$ and $(\wt{Y} \cap C) \bs \wt{Y}$ are abelian;
\item[(3)] $\wt{X}(\BA)$ and $\wt{Y}(\BA)$ preserve $\psi_C$;
\item[(4)] $\psi_C$ is trivial on $(\wt{X} \cap C)(\BA)$ and $(\wt{Y} \cap C)(\BA)$;
\item[(5)] $[\wt{X}, \wt{Y}] \subset C$;
\item[(6)]  there is a non-degenerate pairing $(\wt{X} \cap C)(\BA) \times (\wt{Y} \cap C)(\BA) \rightarrow \BC^*$, given by $(x,y) \mapsto \psi_C([x,y])$, which is
multiplicative in each coordinate, and identifies $(\wt{Y} \cap C)(F) \bs \wt{Y}(F)$ with the dual of
$
\wt{X}(F)(\wt{X} \cap C)(\BA) \bs \wt{X}(\BA),
$
and
$(\wt{X} \cap C)(F) \bs \wt{X}(F)$ with the dual of
$
\wt{Y}(F)(\wt{Y} \cap C)(\BA) \bs \wt{Y}(\BA).
$
\end{itemize}

\begin{rmk}
The basic example of the setting above is the case that $C$ is the center of the Heisenberg group and $\wt{X}, \wt{Y}$ correspond to the two isotropic spaces.
\end{rmk}

Let $B =C\wt{Y}$ and $D=C\wt{X}$, and extend $\psi_C$ trivially to characters of $[B]=B(F)\bs B(\BA)$ and $[D]=D(F)\bs D(\BA)$,
which will be denoted by $\psi_B$ and $\psi_D$ respectively.

\begin{lem}\label{nvequ}
Assume the quadruple $(C, \psi_C, \wt{X}, \wt{Y})$ satisfies the above conditions. Let $f$ be an automorphic form on $G_n(\BA)$. Then
$$\int_{[C]} f(cg) \psi_C^{-1}(c) dc \equiv 0, \forall g \in G_n(\BA),$$
if and only if
$$\int_{[D]} f(ug) \psi_D^{-1}(u) du \equiv 0, \forall g \in G_n(\BA),$$
if and only if
$$\int_{[B]} f(vg) \psi_B^{-1}(v) dv \equiv 0, \forall g \in G_n(\BA).$$
\end{lem}

For simplicity, we will use $\psi_C$ to denote its extensions $\psi_B$ and $\psi_D$
in the remaining of the paper when we use Lemma \ref{nvequ}. Applying Lemma \ref{nvequ} to
the $\psi_{\udl{p}, \underline{a}}$-Fourier coefficients for automorphic forms on $G_n(\BA)$, we obtain a different proof of
a useful result \cite[Lemma 1.1]{GRS03}. For completeness and convenience, we include it here.

\begin{cor}\label{halfheisen}
Let
$
\underline{p}=[p_1^{e_1}p_2^{e_2}\cdots p_r^{e_r}]
$
be a standard symplectic partition of $2n$ with
$p_1\geq p_2\geq\cdots\geq p_r>0$ and $2n=\sum_{i=1}^r e_ip_i$. Then
the $\psi_{\udl{p}, \underline{a}}$-Fourier coefficient $\varphi^{\psi_{\udl{p}, \underline{a}}}$ of an automorphic form
$\varphi$ on $G_n(\BA)$ is non-vanishing if and only if the following integral
\begin{equation*}
\int_{[Y]}\int_{[V_{\udl{p},2}]} \varphi(vyg) \psi_{\udl{p}, \underline{a}}(v)^{-1} dvdy
\end{equation*}
is non-vanishing; and
if and only if the following integral
\begin{equation*}
\int_{[X]}\int_{[V_{\udl{p},2}]} \varphi(vxg) \psi_{\udl{p}, \underline{a}}(v)^{-1} dvdx
\end{equation*}
is non-vanishing.
Here the subgroups $X$ and $Y$ are defined in \eqref{equx} and \eqref{equy}, respectively.
\end{cor}

\begin{proof}
By Lemma \ref{key1} and Corollary \ref{polar}, $V_{\udl{p},1} / \ker_{V_{\udl{p},2}}(\psi_{\udl{p}, \underline{a}})$ has a Heisenberg
structure $W \oplus Z$, where $Z \cong V_{\udl{p},2}/\ker_{V_{\udl{p},2}}(\psi_{\udl{p}, \underline{a}})$,
and $X \oplus Y$ is a polarization of $W$, where $X, Y$ are defined in \eqref{equx}, \eqref{equy}.
This implies directly that the quadruple $(V_{\udl{p},2}, \psi_{\udl{p}, \underline{a}}, X, Y)$
satisfies all the conditions for Lemma \ref{nvequ}.
\end{proof}

\section{Certain Properties of Fourier Coefficients}

In our study of Fourier coefficients of automorphic forms, we have often to consider the Fourier development in stages, which leads
to consider Fourier coefficients attached to composite partitions. We consider two types of composite partitions in Section 3.2,
the proofs of which extend the arguments used in the proofs of \cite[Lemmas 2.5 and 2.6]{GRS03}. For completeness,
we give fully detailed proofs here, including in Section 3.1 a proof for a technical lemma \cite[Lemma 2.6]{GRS03}.
We refer to \cite[Section 1]{GRS03} for definition of composite partitions.
Note that $G_n(\BA)=\Sp_{2n}(\BA)$ or $\tilsp_{2n}(\BA)$.

\subsection{A Technical Lemma}
The technical, but useful lemma we referred here is Lemma 2.6 of \cite{GRS03}. In order to understand the technical insight of the statement of
the lemma and its proof, we provide a repeatable argument based on Lemma \ref{nvequ} and give a complete proof and
more explicit statement.

\begin{lem}[Lemma 2.6, \cite{GRS03}] \label{lem1}
Assume that $\pi$ is an irreducible automorphic representations of
$G_n(\BA)$, and $\ul{p}=[(2n_1)p_2^{e_2}\cdots p_r^{e_r}]$ is a
symplectic partition of $2n$, $2n_1 \geq p_2 \geq \cdots \geq p_r$; $e_i=1$ if $p_i$ is even; and $e_i=2$ if $p_i$ is odd.
Then, $\pi$ has a nonzero $\psi_{\ul{p}, \{a_1\} \cup \ul{a}}$-Fourier
coefficient attached to $\ul{p}$ if and only if
it has a nonzero $\psi_{[(2n_1)1^{2n-2n_1}], a_1} \cdot
\psi_{[p_2^{e_2}\cdots p_r^{e_r}], \ul{a}}$-Fourier
coefficient attached to the composite partition
$[(2n_1)1^{2n-2n_1}] \circ [p_2^{e_2}\cdots p_r^{e_r}]$,
where $a_1\in F^*/(F^*)^2$ and  $\ul{a}= \{a_i\in F^*/(F^*)^2\ |\ e_i =1, 2 \leq i \leq r\}$.
\end{lem}

\begin{proof}
Take $\varphi \in \pi$.
We are going to start from a $\psi_{\ul{p}, \{a_1\} \cup \ul{a}}$-Fourier
coefficient of $\varphi$ attached to $\ul{p}$.
For definition of composite partitions and the attached Fourier coefficients, we
refer to Section 1 of \cite{GRS03}.

By \eqref{fc}, $\varphi^{\psi_{\ul{p}, \{a_1\} \cup \ul{a}}}$
is defined as follows:
\begin{equation}\label{lem1equ1}
\int_{[V_{\ul{p}, 2}]} \varphi(vg) \psi_{\ul{p}, \{a_1\} \cup \ul{a}}^{-1}(v) dv.
\end{equation}

By Corollary \ref{halfheisen} (see also Lemma 1.1 of \cite{GRS03}),
the $\psi_{\ul{p}, \{a_1\} \cup \ul{a}}$-Fourier
coefficient of $\varphi$ in \eqref{lem1equ1} is non-vanishing if and only if
the following integral is non-vanishing:
\begin{equation}\label{lem1equ2}
\int_{[YV_{\ul{p}, 2}]} \varphi(vyg) \psi_{\ul{p}, \{a_1\} \cup \ul{a}}^{-1}(v) dvdy,
\end{equation}
where $Y$ is a group defined in \eqref{equy} corresponding to
the partition $\ul{p}$.

Let $\omega_1$ be a Weyl element of $G_{n-n_1}$ which sends the one-dimensional
toric subgroup $\CH_{[p_2^{e_2}\cdots p_r^{e_r}]}$ defined in \eqref{equtorus}
to the following toric subgroup:
$\{\diag(t^{p_2-1}, \ldots, t^{1-p_2})\}$,
where the exponents of $t$ are of non-increasing order.

Let $\omega_2 = \diag(I_{n_1}; \omega_1; I_{n_1})$.
Conjugating by $\omega_2$, the integral in
\eqref{lem1equ2} becomes:
\begin{equation}\label{lem1equ3}
\int_{[W]} \varphi(w \omega_2 g) \psi_W^{-1}(w) dw,
\end{equation}
where $W= \omega_2 YV_{\ul{p}, 2} \omega_2^{-1}$,
$\psi_W(w)=\psi_{\ul{p}, \{a_1\} \cup \ul{a}}(\omega_2^{-1} w \omega_2)$.
Elements in $W$ have the following form
\begin{equation}\label{lem1equ4}
w=\begin{pmatrix}
z & q_1 & q_2\\
0 & v' & q_1^*\\
0 & 0 & z^*
\end{pmatrix}
\begin{pmatrix}
I_{2k+1} & 0 & 0\\
p_1 & I_{2n-4k-2} & 0\\
0 & p_1^* & I_{2k+1}
\end{pmatrix},
\end{equation}
where $z\in N_{n_1}$; $q_1 \in \Mat_{(n_1) \times (2n-2n_1)}$
with certain conditions;
$q_2 \in \Mat'_{(n_1) \times (n_1)}$, which consists of
matrices $q$ in $\Mat_{(n_1) \times (n_1)}$ satisfying the property
that $q^t v_{n_1} - v_{n_1} q=0$, and $v_{n_1}$ is a matrix only with
ones on the second diagonal;
$p_1 \in \Mat_{(2n-2n_1)\times (n_1)}$
with certain conditions; and finally
$v' \in \omega_1 \wt{Y}V_{[p_2^{e_2}\cdots p_r^{e_r}],2} \omega_1^{-1}$
with $\wt{Y}$ defined in \eqref{equy}
corresponding to the partition $[p_2^{e_2}\cdots p_r^{e_r}]$.
The character is given by
$$
\psi_W(w)
= \psi(\sum_{i=1}^{n_1-1} z_{i,i+1})
\psi(a_1 q_2(n_1, 1))
\psi_{[p_2^{e_2}\cdots p_r^{e_r}],\ul{a}}(\omega_1^{-1} v' \omega_1),
$$
for $w \in W$ as above.

Now, we specify the conditions on $q_1$ and $p_1$ above.
For $q_1$, the last row is zero, assume that $1 \leq i_0 \leq n_1-1$ is the
first row that $q_1$ has zero entries. For $i_0 \leq i \leq n_1-1$,
assume that $\alpha_{i,1}, \alpha_{i,2}, \ldots, \alpha_{i,s_i}$ are
all the roots such that $q_1$ has zero entries at corresponding places
from right to the left with $s_i \in \BZ_{>0}$.
Then $i_0+1$ is the first column that $p_1$ had
non-zero entries. For $i_0 \leq i \leq n_1-1$,
in the column $i+1$,
there are exactly $s_i$ roots $\beta_{i,j}$ with $1 \leq j \leq s_i$,
such that $p_1$ has non-zero entries at corresponding places
from bottom to top.
Finally, for $i_0 \leq i \leq n_1-1$ and
$1 \leq j \leq s_i$, $\alpha_{i,j} + \beta_{i,j} = e_i - e_{i+1}$.

Define $W=\prod_{i=i_0}^{n_1-1} \prod_{j=1}^{s_i} X_{\beta_{i,j}} \wt{W}$
with $\prod_{i=i_0}^{n_1-1} \prod_{j=1}^{s_i} X_{\beta_{i,j}} \cap \wt{W}=\{1\}$ and define
$\psi_{\wt{W}} = \psi_{W}|_{\wt{W}}$. We are getting ready to apply Lemma \ref{nvequ}.

First, for $i=i_0$, we consider the following sequence of quadruples
\begin{align*}
&(\prod_{i=i_0+1}^{n_1-1} \prod_{j=1}^{s_i} X_{\beta_{i,j}} \prod_{j=2}^{s_{i_0}} X_{\beta_{i_0,j}}\wt{W}, \psi_{\wt{W}},
X_{\alpha_{i_0,1}}, X_{\beta_{i_0,1}}),\\
&(X_{\alpha_{i_0,1}}\prod_{i=i_0+1}^{n_1-1} \prod_{j=1}^{s_i} X_{\beta_{i,j}} \prod_{j=3}^{s_{i_0}} X_{\beta_{i_0,j}}\wt{W}, \psi_{\wt{W}},
X_{\alpha_{i_0,2}}, X_{\beta_{i_0,2}}),\\
& \cdots \\
& (\prod_{j=1}^{s_{i_0}-1}X_{\alpha_{i_0,j}}\prod_{i=i_0+1}^{n_1-1} \prod_{j=1}^{s_i} X_{\beta_{i,j}} \wt{W}, \psi_{\wt{W}},
X_{\alpha_{i_0,s_{i_0}}}, X_{\beta_{i_0,s_{i_0}}}).
\end{align*}
Applying Lemma \ref{nvequ} repeatedly, we see that the integral in
\eqref{lem1equ3} is non-vanishing if and only if the following integral is
non-vanishing
\begin{equation}\label{lem1equ5}
\int_{[\prod_{j=1}^{s_{i_0}}X_{\alpha_{i_0,j}}\prod_{i=i_0+1}^{2k} \prod_{j=1}^{s_i} X_{\beta_{i,j}} \wt{W}]} \varphi(xw \omega_2 g) \psi_{\wt{W}}^{-1}(w) dxdw.
\end{equation}

Next, for $i=i_0+1$, we consider the following sequence of quadruples
\begin{align*}
&(\prod_{j=1}^{s_{i_0}}X_{\alpha_{i_0,j}}\prod_{i=i_0+2}^{n_1-1} \prod_{j=1}^{s_i} X_{\beta_{i,j}} \prod_{j=2}^{s_{i_0+1}} X_{\beta_{i_0+1,j}}\wt{W}, \psi_{\wt{W}},
X_{\alpha_{i_0+1,1}}, X_{\beta_{i_0+1,1}}),\\
&(X_{\alpha_{i_0+1,1}}\prod_{j=1}^{s_{i_0}}X_{\alpha_{i_0,j}}\prod_{i=i_0+2}^{n_1-1} \prod_{j=1}^{s_i} X_{\beta_{i,j}} \prod_{j=3}^{s_{i_0+1}} X_{\beta_{i_0+1,j}}\wt{W}, \psi_{\wt{W}},
X_{\alpha_{i_0+1,2}}, X_{\beta_{i_0+1,2}}),\\
& \cdots,\\
&(\prod_{j=1}^{s_{i_0+1}-1}X_{\alpha_{i_0+1,j}}\prod_{j=1}^{s_{i_0}}X_{\alpha_{i_0,j}}
\prod_{i=i_0+2}^{n_1-1} \prod_{j=1}^{s_i} X_{\beta_{i,j}} \prod_{j=3}^{s_{i_0+1}} X_{\beta_{i_0+1,j}}\wt{W},\\
& \psi_{\wt{W}},
X_{\alpha_{i_0+1,s_{i_0+1}}}, X_{\beta_{i_0+1,s_{i_0+1}}}).
\end{align*}
Applying Lemma \ref{nvequ} repeatedly, we see that the integral in
\eqref{lem1equ5} is non-vanishing if and only if the following integral is
non-vanishing
\begin{equation}\label{lem1equ6}
\int_{[\prod_{i=i_0}^{i_0+1}\prod_{j=1}^{s_i}X_{\alpha_{i,j}}\prod_{i=i_0+2}^{n_1-1} \prod_{j=1}^{s_i} X_{\beta_{i,j}} \wt{W}]} \varphi(xw\omega_2g) \psi_{\wt{W}}^{-1}(w) dxdw.
\end{equation}

Then we continue the above procedure to consider the case of $i=i_0+2, i_0+3, \ldots, n_1-2$, and finally for $i=n_1-1$, we need to consider the following quadruples
\begin{align*}
&(\prod_{i=i_0}^{n_1-2}\prod_{j=1}^{s_i}X_{\alpha_{i_0,j}} \prod_{j=2}^{s_{n_1-1}} X_{\beta_{n_1-1,j}}\wt{W}, \psi_{\wt{W}},
X_{\alpha_{n_1-1,1}}, X_{\beta_{n_1-1,1}}),\\
&(X_{\alpha_{n_1-1,1}}\prod_{i=i_0}^{n_1-2}\prod_{j=1}^{s_i}X_{\alpha_{i_0,j}} \prod_{j=3}^{s_{n_1-1}} X_{\beta_{n_1-1,j}}\wt{W}, \psi_{\wt{W}},
X_{\alpha_{n_1-1,2}}, X_{\beta_{n_1-1,2}}),\\
& \cdots\\
& (\prod_{j=1}^{s_{n_1-1}-1}X_{\alpha_{n_1-1,j}}\prod_{i=i_0}^{n_1-2}\prod_{j=1}^{s_i}
X_{\alpha_{i_0,j}}  \wt{W}, \psi_{\wt{W}},
X_{\alpha_{n_1-1,s_{n_1-1}}}, X_{\beta_{n_1-1,s_{n_1-1}}}).
\end{align*}
Again, applying Lemma \ref{nvequ} repeatedly, we see that the integral in
\eqref{lem1equ6} is non-vanishing if and only if the following integral is
non-vanishing
\begin{equation}\label{lem1equ7}
\int_{[\prod_{i=i_0}^{n_1-1}\prod_{j=1}^{s_i}X_{\alpha_{i,j}}\wt{W}]} \varphi(xw\omega_2g) \psi_{\wt{W}}^{-1}(w) dxdw,
\end{equation}
which is exactly the following integral
\begin{align}\label{lem1equ8}
\begin{split}
& \int_{[\omega_1 \wt{Y}V_{[p_2^{e_2}\cdots p_r^{e_r}],2} \omega_1^{-1}]}
\int_{[\ol{Y} V_{[(2n_1)1^{2n-2n_1}],2}]}
\varphi(v_1 v_2 \omega_2 g) \\
& \psi_{[(2n_1)1^{2n-2n_1}], a_1}(v_2) \cdot
\psi_{[p_2^{e_2}\cdots p_r^{e_r}], \ul{a}}(\omega_1^{-1} v_1 \omega_1)
dv_1 dv_2,
\end{split}
\end{align}
where $v_1 \in \omega_1 \wt{Y}V_{[p_2^{e_2}\cdots p_r^{e_r}],2} \omega_1^{-1}$,
embedded in $G_n$ by $v_1 \mapsto \diag(I_{n_1}, v_1, I_{n_1})$,
and identified with the image, and
$\ol{Y}$ is the group defined in \eqref{equy}
corresponding to the partition $[(2n_1)1^{2n-2n_1}]$.

By the discussion at the end of Section 1 of
\cite{GRS03}, it is easy to see that the integral in
\ref{lem1equ8} is non-vanishing if and only if a
$\psi_{[(2n_1)1^{2n-2n_1}], a_1} \cdot
\psi_{[p_2^{e_2}\cdots p_r^{e_r}], \ul{a}}$-Fourier
coefficient defined in (1.4) of \cite{GRS03} is non-vanishing.

This completes the proof of the lemma.
\end{proof}

\subsection{Composite partitions and Fourier coefficients}
We consider Fourier coefficients of automorphic representations attached to two types of composite partitions, which are stated
as two propositions.

\begin{prop}\label{prop1}
Let $\pi$ be an irreducible automorphic representation of $G_n(\BA)$ realized in the space of automorphic forms.
Assume that $\pi$ has no nonzero Fourier coefficients attached the partition $[(2k+2)1^{2n-2k-2}]$.
Then the following are equivalent:
\begin{enumerate}
\item[(1)] $\pi$ has a nonvanishing Fourier coefficient attached to the composite partition
$[(2k)1^{2n-2k}] \circ [(2k+2)1^{2n-4k-2}]$
with respect to the characters $\psi_{[(2k)1^{2n-2k}],\alpha}$
and $\psi_{[(2k+2)1^{2n-4k-2}], \beta}$, where
$\alpha, \beta \in F^*/(F^*)^2$;
\item[(2)] $\alpha$ and $\beta$ are related by $\beta = -\alpha \mod (F^*)^2$, and
$\pi$ has a nonvanishing Fourier coefficient attached to the partition $[(2k+1)^21^{2n-4k-2}]$.
\end{enumerate}
\end{prop}

\begin{proof}
Assume that $\pi$ has a nonvanishing Fourier coefficient attached to the composite partition
$[(2k)1^{2n-2k}] \circ [(2k+2)1^{2n-4k-2}]$ with respect to the characters $\psi_{[(2k)1^{2n-2k}],\alpha}$
and $\psi_{[(2k+2)1^{2n-4k-2}], \beta}$,
where $\alpha, \beta \in F^*/(F^*)^2$, by definition (\cite[Section 1]{GRS03}) and Corollary \ref{halfheisen},
there is a $\varphi \in \pi$ such that the following integral is nonvanishing:
\begin{align}\label{thm1equ1}
\begin{split}
& \int_{[Y'V_{[(2k+2)1^{2n-4k-2}], 2}]}
\int_{[YV_{[(2k)1^{2n-2k}],2}]}
\varphi(v_1yv_2y'g) \\
& \psi^{-1}_{[(2k)1^{2n-2k}], \alpha}(v_1)
\psi^{-1}_{[(2k+2)1^{2n-4k-2}], \beta}(v_2)
dv_1 dy dv_2 dy',
\end{split}
\end{align}
where $Y$ is the group defined in \eqref{equy} corresponding to the partition
$[(2k)1^{2n-2k}]$, $Y'$ is the group defined in \eqref{equy} corresponding to the partition
$[(2k+2)1^{2n-4k-2}]$, and $Y'V_{[(2k+2)1^{2n-4k-2}], 2}$ is embedded
into $\Sp_{2n}$ by the map $v \mapsto \diag(I_k, v, I_k)$.

Let $V:=Y'V_{[(2k+2)1^{2n-4k-2}], 2} YV_{[(2k)1^{2n-2k}],2}$,
and for $v =v_1 y v_2y' \in V$, define
$$
\psi_V(v)=\psi_{[(2k)1^{2n-2k}], \alpha}(v_1)
\psi_{[(2k+2)1^{2n-4k-2}], \beta}(v_2).
$$
Then the integral in \eqref{thm1equ1} becomes:
\begin{equation}\label{thm1equ4}
\int_{[V]} \varphi(vg) \psi_V^{-1}(v) dv.
\end{equation}
Following \eqref{equtorus}, we define the following one-dimensional
toric subgroup $\CH:=\CH_{[(2k)1^{2n-2k}] \circ [(2k+2)1^{2n-4k-2}]}$
corresponding to the composite partition $[(2k)1^{2n-2k}] \circ [(2k+2)1^{2n-4k-2}]$:
\begin{equation}\label{thm1equ2}
\CH(t)=\diag(T_1; T_2; I_{2n-4k-2}; T_2^*; T_1^*),
\end{equation}
where $T_1=\diag(t_1^{2k-1}, t_1^{2k-3}, \ldots, t_1)$,
$T_2=\diag(t_2^{2k+1}, t_2^{2k-1}, \ldots, t_2)$.
Note that actually here $t_1 = t_2=t$, we just want to label them
to make the definition of certain Weyl element easier.

Let $\omega$ be a Weyl element which sends $\CH$ to the following toric subgroup:
\begin{equation}\label{thm1equ3}
\{\diag(T_3; I_{2n-4k-2}; T_3^*)\},
\end{equation}
where
$$
T_3=\diag(t_2^{2k+1}; t_2^{2k-1}, t_1^{2k-1}; t_2^{2k-3}, t_1^{2k-3};
\ldots; t_2^3, t_1^3; t_2, t_1).
$$
Conjugating by $\omega$, the integral in \eqref{thm1equ4} becomes
\begin{equation}\label{thm1equ5}
\int_{[V^{\omega}]} \varphi(v \omega g) \psi_{V^{\omega}}^{-1}(v) dv,
\end{equation}
where $V^{\omega} = \omega V \omega^{-1}$,
and for $v \in V^{\omega}$, $\psi_{V^{\omega}}(v)=\psi_V(\omega^{-1} v \omega)$.
Elements $v \in V^{\omega}$ have the following form:
\begin{equation}\label{thm1equ6}
\begin{pmatrix}
z & q_1 & q_2\\
0 & I_{2n-4k-2} & q_1^*\\
0 & 0 & z^*
\end{pmatrix},
\end{equation}
where $z$ is in the unipotent radical of the
parabolic subgroup of $\GL_{2k+1}$
with Levi isomorphic to $\GL_1 \times \GL_2 \times \cdots \times \GL_2$
($k$-copies of $\GL_2$), but $z_{1,3}=0$;
$q_1 \in \Mat_{(2k+1) \times (2n-4k-2)}$
with $q_1(i,j)=0$ for any $i=2k, 2k+1$ and $j=1,2,\ldots, n-2k-1$; and finally
$q_2 \in \Mat'_{(2k+1) \times (2k+1)}$, which
is the set of matrices in $\Mat_{(2k+1) \times (2k+1)}$
with property that $q_2^t v_{2k+1} - v_{2k+1} q_2 =0$, where
$v_{2k+1}$ is a matrix only with ones on the second diagonal.
The character is given by
\begin{align}\label{thm1equ7}
\begin{split}
\psi_{V^{\omega}}(v)
= & \psi(z_{1,2})
\cdot\psi(\sum_{i=1}^{k-1} z_{2(i-1)+2, 2(i-1)+4} +
z_{2(i-1)+3, 2(i-1)+5})\\
&\cdot\psi(\beta q_2(2k,2) + \alpha q_2(2k+1,1)).
\end{split}
\end{align}
for $v\in V^\omega$ as in \eqref{thm1equ6}.

Let $X_{e_1-e_3}$ be the root subgroup corresponding to the
root $e_1-e_3$. It follows from the structure of
elements in $V^{\omega}$ that $X_{e_1-e_3}$ normalizes
$V^{\omega}$. Hence we are able to take the Fourier
expansion of the integral in \eqref{thm1equ5}
along $X_{e_1-e_3}$. After the Fourier expansion,
the integral in \eqref{thm1equ5} becomes
\begin{align}\label{thm1equ8}
\begin{split}
& \sum_{\gamma \in F^*} \int_{[X_{e_1-e_3}]}
\int_{[V^{\omega}]} \varphi(v x \iota(\gamma) \omega g) \psi_{V^{\omega}}^{-1}(v)
\psi^{-1}(x) dv dx\\
+ & \int_{[X_{e_1-e_3}]}
\int_{[V^{\omega}]} \varphi(v x \omega g) \psi_{V^{\omega}}^{-1}(v)
 dv dx,
\end{split}
\end{align}
where $\iota(\gamma) = \diag(\gamma, I_{2n-2}, \gamma^{-1})$.
Note that
up to conjugation by a Weyl element, the integral
$$
\int_{[X_{e_1-e_3}]}
\int_{[V^{\omega}]} \varphi(v x \omega g) \psi_{V^{\omega}}^{-1}(v)dv dx
$$
is essentially a Fourier coefficient attached to $[(2k+2)(2k)1^{2n-4k-2}]$, which
is identically zero, since $\pi$ has no nonzero Fourier coefficients attached the partition $[(2k+2)1^{2n-2k-2}]$.
Therefore, the integral in \eqref{thm1equ5} is nonvanshing if and only if there is at least one $\gamma \in F^*$ such that
the following integral is nonvanishing:
\begin{equation}\label{thm1equ9}
\int_{[X_{e_1-e_3}]}
\int_{[V^{\omega}]} \varphi(v x \iota(\gamma) \omega g) \psi_{V^{\omega}}^{-1}(v)
\psi^{-1}(x) dv dx.
\end{equation}
Let $U = V^{\omega} X_{e_1-e_3}$, and for $u=vx \in U$,
define $\psi_U(vx) = \psi_{V^{\omega}}(v)
\psi(x)$. The integral in \eqref{thm1equ9}
becomes
\begin{equation}\label{thm1equ10}
\int_{[U]}
\varphi(u \iota(\gamma) \omega g) \psi_{U}^{-1}(u)
du.
\end{equation}

If $\beta \neq -\alpha$ mod $(F^*)^2$ in the character $\psi_{V^{\omega}}$ as expressed in \eqref{thm1equ7}, take
$A = \frac{1}{\alpha+\beta} \begin{pmatrix}
\beta & \alpha\\
-1 & 1
\end{pmatrix}$, and define
$$
\epsilon = \diag(1, A, \ldots, A, I_{2n-4k-2}, A^*, \ldots, A^*, 1).
$$
Conjugating by $\epsilon$, the integral
in \eqref{thm1equ10} becomes:
\begin{equation}\label{thm1equ11}
\int_{[U^{\epsilon}]}
\varphi(u \epsilon \iota(\gamma) \omega g) \psi_{U^{\epsilon}}^{-1}(u)
du,
\end{equation}
where $U^{\epsilon}:= \epsilon U \epsilon^{-1}$, and
for $u \in U^{\epsilon}$, $\psi_{U^{\epsilon}}(u) = \psi_U(\epsilon^{-1} u \epsilon)$.
Elements $u \in U^{\epsilon}$ have the following form:
\begin{equation}\label{thm1equ12}
\begin{pmatrix}
z & q_1 & q_2\\
0 & I_{2n-4k-2} & q_1^*\\
0 & 0 & z^*
\end{pmatrix},
\end{equation}
with $z$ in the unipotent radical of the
parabolic subgroup of $\GL_{2k+1}$
whose Levi isomorphic to $\GL_1 \times \GL_2 \times \cdots \times \GL_2$
($k$-copies of $\GL_2$) and with
$q_1, q_2$ having the same structure as elements in \eqref{thm1equ6}.
The character is given by
\begin{align}\label{thm1equ13}
\begin{split}
\psi_{U^{\epsilon}}(u)
= & \psi(z_{1,2})\cdot\psi(\sum_{i=1}^{k-1} z_{2(i-1)+2, 2(i-1)+4} +
z_{2(i-1)+3, 2(i-1)+5})\\
& \cdot \psi((\alpha+\beta) q_2(2k,2) + \alpha\beta(\alpha+\beta) q_2(2k+1,1))
\end{split}
\end{align}
for $u\in U^\epsilon$ as in \eqref{thm1equ12}.
Therefore, the integral in \eqref{thm1equ11}
is a Fourier coefficient attached to
the partition $[(2k+2)(2k)1^{2n-4k-2}]$, which
is identically zero by assumption.

Therefore, we may only consider the case $\beta = - \alpha$ mod $(F^*)^2$ in the character $\psi_{V^{\omega}}$ as expressed in \eqref{thm1equ7}. Without loss of generality, we may assume that $\beta = -\alpha$. 

Now, take $B = \frac{1}{2} \begin{pmatrix}
1 & -1\\
1 & 1
\end{pmatrix}$ and define
$$
\ol{\epsilon}: = \diag(1, B, \ldots, B, I_{2n-4k-2}, B^*, \ldots, B^*, 1).
$$
Conjugating by $\ol{\epsilon}$, the integral
in \eqref{thm1equ10} becomes:
\begin{equation}\label{thm1equ18}
\int_{[U^{\ol{\epsilon}}]}
\varphi(u \ol{\epsilon}\iota(\gamma) \omega g) \psi_{U^{\ol{\epsilon}}}^{-1}(u)
du,
\end{equation}
where $U^{\ol{\epsilon}}:= \ol{\epsilon} U \ol{\epsilon}^{-1}$, and
for $u \in U^{\ol{\epsilon}}$, $\psi_{U^{\ol{\epsilon}}}(u) = \psi_U(\ol{\epsilon}^{-1} u \ol{\epsilon})$.
Elements $u \in U^{\ol{\epsilon}}$ have the following form:
\begin{equation}\label{thm1equ19}
\begin{pmatrix}
z & q_1 & q_2\\
0 & I_{2n-4k-2} & q_1^*\\
0 & 0 & z^*
\end{pmatrix},
\end{equation}
with the same structure as elements in \eqref{thm1equ12}.
The character is given by
\begin{align}\label{thm1equ20}
\begin{split}
\psi_{U^{\ol{\epsilon}}}(u)
= & \psi(z_{1,3})\cdot\psi(\sum_{i=1}^{k-1} z_{2(i-1)+2, 2(i-1)+4} +
z_{2(i-1)+3, 2(i-1)+5})\\
&\cdot\psi(-4\alpha q_2(2k,1)).
\end{split}
\end{align}
for $u\in U^{\ol{\epsilon}}$ as defined in \eqref{thm1equ19}.
Let $\omega'$ be a Weyl element which sends
the one-dimensional toric subgroup in \eqref{thm1equ3}
to the following one-dimensional toric subgroup:
\begin{equation}\label{thm1equ14}
\{\diag(T_4; I_{2n-4k-2}; T_5)\},
\end{equation}
where
$$T_4=\diag(t_2^{2k+1}; t_1^{2k-1}, t_1^{2k-3}, \ldots, t_1; t_2^{-1}, t_2^{-3},
\ldots, t_2^{1-2k}),$$
and
$$T_5=\diag(t_2^{2k-1}, t_2^{2k-3}, \ldots, t_2; t_1^{-1}, t_1^{-3},
\ldots, t_1^{1-2k}; t_2^{-2k-1}).$$
Conjugating by $\omega'$, the integral in \eqref{thm1equ18}
becomes:
\begin{equation}\label{thm1equ15}
\int_{[U^{\ol{\epsilon}, \omega'}]}
\varphi(u \omega' \ol{\epsilon} \iota(\gamma) \omega g) \psi_{U^{\ol{\epsilon}, \omega'}}^{-1}(u)
du,
\end{equation}
where $U^{\ol{\epsilon}, \omega'}:= \omega' U^{\ol{\epsilon}} {\omega'}^{-1}$, and
for $u \in U^{\ol{\epsilon}, \omega'}$,
$\psi_{U^{\ol{\epsilon}, \omega'}}(u) = \psi_{U^{\ol{\epsilon}}}({\omega'}^{-1} u \omega')$.
Elements $u \in U^{\ol{\epsilon}, \omega'}$ have the following form:
\begin{equation}\label{thm1equ16}
\begin{pmatrix}
z & q_1 & q_2\\
0 & I_{2n-4k-2} & q_1^*\\
0 & 0 & z^*
\end{pmatrix}
\begin{pmatrix}
I_{2k+1} & 0 & 0\\
p_1 & I_{2n-4k-2} & 0\\
p_2 & p_1^* & I_{2k+1}
\end{pmatrix},
\end{equation}
where
$z\in N_{2k+1}$, the maximal upper-triangular unipotent subgroup of $\GL_{2k+1}$;
$q_1 \in \Mat_{(2k+1) \times (2n-4k-2)}$  with
$q_1(i,j)=0$ for any $i \geq k+2$ and for any
$i=k+1$ and $1 \leq j \leq n-2k-1$;
$q_2 \in \Mat'_{(2k+1) \times (2k+1)}$
with $q_2(i,j)=0$ for $i > j$;
$p_1 \in \Mat_{(2n-4k-2) \times (2k+1)}$
with $p_1(i,j)=0$ for any $j \leq k+1$
and for any $i \geq n-2k$ and $j=k+2$; and
$p_2 \in \Mat'_{(2k+1) \times (2k+1)}$
with $p_2(i,j)=0$ for $i \geq \max(j-1,1)$.
The character is given by
\begin{align}\label{thm1equ17}
\begin{split}
\psi_{U^{\ol{\epsilon}, \omega'}}(u)
= \psi(\sum_{i=1}^{k-1} z_{i,i+1}-4\alpha z_{k,k+1}-\sum_{i=k+1}^{2k} z_{i,i+1}),
\end{split}
\end{align}
for $u$ as in \eqref{thm1equ16}.

It is easy to see that the integral in \eqref{thm1equ15} is non-vanishing
if and only if the following integral is non-vanishing:
\begin{equation}\label{thm1equ21}
\int_{[U^{\ol{\epsilon}, \omega'}]}
\varphi(u \omega' \ol{\epsilon} \iota(\gamma) \omega g) \psi_{U^{\ol{\epsilon}, \omega'}}^{-1}(u)
du,
\end{equation}
with
\begin{align}\label{thm1equ22}
\begin{split}
\psi_{U^{\ol{\epsilon}, \omega'}}(u)
= \psi(\sum_{i=1}^{2k} z_{i,i+1}),
\end{split}
\end{align}
for $u$ as in \eqref{thm1equ16}.

We claim that the integral in \eqref{thm1equ21} is nonzero if and only if
$\varphi$ has a nonzero Fourier coefficient attached to
$[(2k+1)^2 1^{2n-4k-2}]$.

Indeed, if $\varphi$ has a non-zero Fourier coefficient attached to the partition
$[(2k+1)^2 1^{2n-4k-2}]$, then
the following integral is nonzero:
\begin{equation}\label{thm1equ25}
\int_{[V_{[(2k+1)^2 1^{2n-4k-2}],2}]}
\varphi(vg) \psi_{[(2k+1)^2 1^{2n-4k-2}]}^{-1}(v) dv,
\end{equation}
where elements in $[V_{[(2k+1)^2 1^{2n-4k-2}],2}]$ have
the following form
\begin{equation}\label{thm1equ23}
\begin{pmatrix}
z & q_1 & q_2\\
0 & I_{2n-4k-2} & q_1^*\\
0 & 0 & z^*
\end{pmatrix}
\begin{pmatrix}
I_{2k+1} & 0 & 0\\
p_1 & I_{2n-4k-2} & 0\\
p_2 & p_1^* & I_{2k+1}
\end{pmatrix},
\end{equation}
where in the left matrix, $z\in N_{2k+1}$;
$q_1 \in \Mat_{(2k+1) \times (2n-4k-2)}$ with
$q_1(i,j)=0$ for any $i \geq k+1$; and
$q_2 \in \Mat'_{(2k+1) \times (2k+1)}$
with $q_2(i,j)=0$ for $i \geq j$, while in the right matrix,
$p_1 \in \Mat_{(2n-4k-2) \times (2k+1)}$
with $p_1(i,j)=0$ for any $j \leq k+1$; and
$p_2 \in \Mat'_{(2k+1) \times (2k+1)}$
with $p_2(i,j)=0$, for $i \geq j$.
The character is given by
\begin{align}\label{thm1equ24}
\begin{split}
\psi_{[(2k+1)^2 1^{2n-4k-2}]}(v)
= \psi(\sum_{i=1}^{2k} z_{i,i+1}),
\end{split}
\end{align}
for $v$ as in \eqref{thm1equ23}.

Define $R=\prod_{i=n+1}^{2n-2k-1} X_{\alpha_i}$, where
$X_{\alpha_i}$ is the root subgroup corresponding to the
root $\alpha_i = e_{k+1}+e_i$.
Let $C=\prod_{i=n+1}^{n-2k-1} X_{\beta_i}$,
where $X_{\beta_i}$ is the root subgroup corresponding to the
root $\beta_i = -e_i-e_{k+2}$.
Let $\gamma_i= e_i+e_{2k+2-i}$, $\eta_i = -e_{2k+2-i}-e_{i+1}$,
for $i=1,2,\ldots,k$, and $x_{\gamma_i}$, $X_{\eta_i}$ are
corresponding root subgroups.
Write $V_{[(2k+1)^2 1^{2n-4k-2}],2}=C \prod_{i=1}^k X_{\eta_i} \wt{V}$,
with $C \prod_{i=1}^k X_{\eta_i} \cap \wt{V} = \{1\}$ and define $\psi_{\wt{V}} = \psi_{[(2k+1)^2 1^{2n-4k-2}]}|_{\wt{V}}$.

Consider the quadruple $(\prod_{i=1}^k X_{\eta_i} \wt{V}, \psi_{\wt{V}}, R, C)$.
It is easy to see that it satisfies the conditions for Lemma
\ref{nvequ} and hence
the integral in \eqref{thm1equ25} is non-vanishing if and only if
the following integral is non-vanishing
\begin{equation}\label{thm1equ26}
\int_{[R\prod_{i=1}^k X_{\eta_i} \wt{V}]}
\varphi(xvg) \psi_{\wt{V}}^{-1}(v) dxdv.
\end{equation}

Then consider the following sequence of quadruples
\begin{align}\label{thm1equ27}
\begin{split}
& (\prod_{i=2}^k X_{\eta_i} R \wt{V}, \psi_{\wt{V}}, X_{\gamma_1}, X_{\eta_1}),\\
& (X_{\gamma_1}\prod_{i=3}^k X_{\eta_i} R \wt{V}, \psi_{\wt{V}}, X_{\gamma_2}, X_{\eta_2}),\\
& \cdots \\
& (\prod_{i=1}^{k-1} X_{\gamma_i} R \wt{V}, \psi_{\wt{V}}, X_{\gamma_k}, X_{\eta_k}).
\end{split}
\end{align}
By applying Lemma \ref{nvequ} repeatedly, the integral
in \eqref{thm1equ26} is non-vanishing if and only if the
following integral is non-vanishing
\begin{equation}\label{thm1equ28}
\int_{[\prod_{i=1}^k X_{\gamma_i} R \wt{V}]}
\varphi(xvg) \psi_{\wt{V}}^{-1}(v) dxdv.
\end{equation}

Since the root subgroup
$X_{e_{k+1} + e_{k+1}}$ normalizes the group
$\prod_{i=1}^k X_{\gamma_i} R \wt{V}$,
we can take the Fourier expansion of the integral
in \eqref{thm1equ28} along the subgroup $X_{e_{k+1} + e_{k+1}}$.
We will get two orbits. Any Fourier coefficient corresponding to
the non-trivial orbit contains an inner integral which is
a Fourier coefficient attached to the partition
$[(2k+2)1^{2n-2k-2}]$, and is identically zero by assumption.
Therefore, the integral
in \eqref{thm1equ28} becomes
 \begin{equation}\label{thm1equ29}
\int_{[X_{e_{k+1} + e_{k+1}}\prod_{i=1}^k X_{\gamma_i} R \wt{V}]}
\varphi(xvg) \psi_{\wt{V}}^{-1}(v) dxdv.
\end{equation}
It is easy to see that the integral in \eqref{thm1equ29} is exactly the integral
in \eqref{thm1equ21}. This proves the claim and completes the proof of the proposition.
\end{proof}

For the partition $[(2k+1)^21^{2n-4k-2}]$, the Levi part of the stabilizer of the
character $\psi_{[(2k+1)^21^{2n-4k-2}]}$ in
is isomorphic to
$G_1(\BA) \times G_{n-2k-1}(\BA)$.
It follows that for any automorphic form $\varphi$, the Fourier coefficient
$\varphi^{\psi_{[(2k+1)^21^{2n-4k-2}]}}$ is automorphic
on $G_1(\BA) \times G_{n-2k-1}(\BA)$, and in particular is automorphic on
$G_{n-2k-1}(\BA)$ by restriction to $G_{n-2k-1}(\BA)$.
Hence, we may define Fourier coefficients attached to the composite partition
$[(2k+1)^21^{2n-4k-2}]\circ[p_1^{e_1}p_2^{e_2}\cdots p_r^{e_r}]$
for any symplectic partition $[p_1^{e_1}p_2^{e_2}\cdots p_r^{e_r}]$ of $2n-4k-2$ as in Section 1 of \cite{GRS03}.
Note that $G_{n-2k-1}(\BA)$ is embedded into $G_n(\BA)$ via the map
$g \mapsto \diag(I_{2k+1}, g, I_{2k+1})$.
The following proposition is an analogue of Lemma 2.6 of \cite{GRS03}.

\begin{prop}\label{prop2}
Assume that $\pi$ is an irreducible automorphic representation of
$G_n(\BA)$ realized in the space of automorphic forms, and $\ul{p}=[(2k+1)^2p_1^{e_1}p_2^{e_2}\cdots p_r^{e_r}]$ is a symplectic partition of $2n$ with
$2k+1 \geq p_1 \geq p_2 \geq \cdots \geq p_r$;
$e_i=1$ if $p_i$ is even; and $e_i=2$ if $p_i$ is odd.
Then $\pi$ has a non-vanishing $\psi_{\ul{p}, \ul{a}}$-Fourier coefficient attached to $\ul{p}$ if and only if it has a non-vanishing $\psi_{[(2k+1)^21^{2n-4k-2}]} \cdot \psi_{[p_1^{e_1}p_2^{e_2}\cdots p_r^{e_r}], \ul{a}}$-
Fourier coefficient attached to
the composite partition $[(2k+1)^21^{2n-4k-2}]\circ[p_1^{e_1}p_2^{e_2}\cdots p_r^{e_r}]$, where $\ul{a}= \{a_i\in F^*/(F^*)^2\ |\ e_i =1, 2 \leq i \leq r\}$.
\end{prop}

\begin{proof}
We are going to show that for any $\varphi \in \pi$,
both Fourier coefficients attached to
partitions $[(2k+1)^21^{2n-4k-2}]\circ[p_1^{e_1}p_2^{e_2}\cdots p_r^{e_r}]$
and $[(2k+1)^2p_1^{e_1}p_2^{e_2}\cdots p_r^{e_r}]$ share the same non-vanishing
property with the following kind of integrals
\begin{equation}\label{thm2equ1}
\int_{[U]} \varphi(ug) \psi_U^{-1}(u) du,
\end{equation}
where elements in $U$ have the following form
\begin{equation}\label{thm2equ2}
\begin{pmatrix}
z & q_1 & q_2\\
0 & v' & q_1^*\\
0 & 0 & z^*
\end{pmatrix}
\begin{pmatrix}
I_{2k+1} & 0 & 0\\
p_1 & I_{2n-4k-2} & 0\\
p_2 & p_1^* & I_{2k+1}
\end{pmatrix},
\end{equation}
where $z\in N_{2k+1}$; $q_1 \in \Mat_{(2k+1) \times (2n-4k-2)}$
with $q_1(i, j)=0$ for $k+1 \leq i \leq 2k+1$ and $1 \leq j \leq 2n-4k-2$; $p_1 \in \Mat_{(2n-4k-2) \times (2k+1)}$
with $p_1(i, j)=0$ for $1 \leq i \leq 2n-4k-2$ and $1 \leq j \leq k+1$; and
$p_2, q_2 \in \Mat'_{(2k+1) \times (2k+1)}$ with
$q_2(i,j)=0$ for $k+1 \leq i \leq 2k+1$ and $1 \leq j \leq k+1$, $p_2 \neq 0$ for $1 \leq i \leq k$ and $k+2 \leq j \leq 2k+1$. Recall that
$\Mat'_{(2k+1) \times (2k+1)}$ consists of matrices
in $\Mat_{(2k+1) \times (2k+1)}$ satisfying
the property that $q^t v_{2k+1} - v_{2k+1} q =0$, where
$v_{2k+1}$ is a matrix only with ones on the second diagonal.
Also we have that
$v' \in \omega_1 \wt{Y} V_{[p_1^{e_1}p_2^{e_2}\cdots p_r^{e_r}],2} \omega_1^{-1}$,
where $\wt{Y}$ is defined in \eqref{equy}
corresponding to the partition $[p_1^{e_1}p_2^{e_2}\cdots p_r^{e_r}]$,
and $\omega_1$ is any Weyl element sending the one
dimensional toric subgroup $\CH_{[p_1^{e_1}p_2^{e_2}\cdots p_r^{e_r}]}$
attached to the partition $[p_1^{e_1}p_2^{e_2}\cdots p_r^{e_r}]$
(see \eqref{equtorus})
to the toric subgroup:
$\{\diag(t^{p_1-1}, \ldots, t^{1-p_1})\},$
with the exponents of $t$ being of non-increasing order.
The character is given by
$$
\psi_U(\diag(z, v', z^*))
=  \psi(\sum_{i=1}^{2k} z_{i,i+1})\cdot
\psi_{[p_1^{e_1}p_2^{e_2}\cdots p_r^{e_r}],\ul{a}}(\omega_1^{-1} v' \omega_1).
$$

Without loss of generality, we assume that $[p_1^{e_1}p_2^{e_2}\cdots p_r^{e_r}] \neq [1^{2n-4k-2}]$.

We fix any automorphic form $\varphi \in \pi$ and start from its $\psi_{[(2k+1)^21^{2n-4k-2}]} \cdot \psi_{[p_1^{e_1}p_2^{e_2}\cdots p_r^{e_r}], \ul{a}}$-Fourier coefficients attached to the composite
partition $[(2k+1)^21^{2n-4k-2}]\circ[p_1^{e_1}p_2^{e_2}\cdots p_r^{e_r}]$.
After adding $\wt{Y}$ using Corollary \ref{halfheisen} and
conjugating by the Weyl element
$\diag(I_{2k+1}, \omega_1, I_{2k+1})$,
a Fourier coefficient of $\varphi$ attached to
the partition $[(2k+1)^21^{2n-4k-2}]\circ[p_1^{e_1}p_2^{e_2}\cdots p_r^{e_r}]$
can be taken to be following form
\begin{equation}\label{thm2equ4}
\int_{[V]} \varphi(vg) \psi_V^{-1}(v) dv,
\end{equation}
where elements in $V$ have the following form
\begin{equation}\label{thm2equ5}
\begin{pmatrix}
z & q_1 & q_2\\
0 & v' & q_1^*\\
0 & 0 & z^*
\end{pmatrix}
\begin{pmatrix}
I_{2k+1} & 0 & 0\\
p_1 & I_{2n-4k-2} & 0\\
p_2 & p_1^* & I_{2k+1}
\end{pmatrix},
\end{equation}
where $z\in N_{2k+1}$; $q_1 \in \Mat_{(2k+1) \times (2n-4k-2)}$
with $q_1(i, j)=0$ for any $i \geq k+1$ and $1 \leq j \leq 2n-4k-2$; and
$p_2, q_2 \in \Mat'_{(2k+1) \times (2k+1)}$ with
$p_2(i,j)=q_2(i,j)=0$, for any $i \geq j$. Recall that
$\Mat'_{(2k+1) \times (2k+1)}$ consists of matrices
in $\Mat_{(2k+1) \times (2k+1)}$ satisfying
the property that $q^t v_{2k+1} - v_{2k+1} q =0$, with
$v_{2k+1}$ being a matrix only with ones on the second diagonal. Also
$p_1 \in \Mat_{(2n-4k-2)\times (2k+1)}$
with $p_1(i, j)=0$ for any $1 \leq i \leq 2n-4k-2$ and $1 \leq j \leq k+1$; and
$v' \in \omega_1 \wt{Y} V_{[p_1^{e_1}p_2^{e_2}\cdots p_r^{e_r}],2} \omega_1^{-1}$.
The character is given by
$$
\psi_V(\diag(z, v', z^*))
=  \psi(\sum_{i=1}^{2k} z_{i,i+1})
\psi_{[p_1^{e_1}p_2^{e_2}\cdots p_r^{e_r}],\ul{a}}(\omega_1^{-1} v' \omega_1).
$$

For $1 \leq i \leq k$, we define $R^2_i = \prod_{j=1}^i X_{\alpha^i_j}$,
with $\alpha^i_j = e_i+e_{2k+1-i+j}$,
and $C^2_i = \prod_{j=1}^i X_{\beta^i_j}$,
with $\beta^i_j = -e_{2k+1-i+j}-e_{i+1}$.
Then define that $V = \prod_{i=1}^{k} C_i^2 \wt{V}$
with $\prod_{i=1}^{k} C_i^2 \cap \wt{V} = \{1\}$ and
$\psi_{\wt{V}} = \psi_V|_{\wt{V}}$.

We will apply Lemma \ref{nvequ} repeatedly to the integration over
$\prod_{i=1}^{k} C_i^2.$
To do this, we need to consider pairs of groups ordered as following:
$$
(R_1^2, C_1^2), (R_2^2, C_2^2), \ldots, (R_k^2, C_k^2).
$$

For the pair $(R_1^2, C_1^2)$, we consider the following sequence of quadruples
\begin{align*}
& (\prod_{i=2}^{k} C_i^2 \wt{V}, \psi_{\wt{V}}, X_{\alpha_1^1}, X_{\beta_1^1}).
\end{align*}
By Lemma \ref{nvequ}, the integral in
\eqref{thm2equ4} is non-vanishing if and only if the following integral is non-vanishing
\begin{equation}\label{thm2equ12}
\int_{[R_1^2\prod_{i=2}^{k} C_i^2 \wt{V}]} \varphi(rvg) \psi_{\wt{V}}^{-1}(v) dr dv.
\end{equation}

For the pair $(R_s^2, C_s^2)$, $s=2,3,\ldots, k$, we consider the following sequence of quadruples
\begin{align*}
& (\prod_{t=2}^s X_{\beta_t^s} \prod_{j=1}^{s-1} R_j^2 \prod_{i=s+1}^{k} C_i^2 \wt{V}, \psi_{\wt{V}}, X_{\alpha_1^s}, X_{\beta_1^s}),\\
& (\prod_{t=3}^s X_{\beta_t^s} X_{\alpha_1^s} \prod_{j=1}^{s-1} R_j^2 \prod_{i=s+1}^{k} C_i^2 \wt{V}, \psi_{\wt{V}}, X_{\alpha_2^s}, X_{\beta_2^s}),\\
& \cdots,\\
& (\prod_{t=1}^{s-1} X_{\alpha_t^s} \prod_{j=1}^{s-1} R_j^2 \prod_{i=s+1}^{k} C_i^2 \wt{V}, \psi_{\wt{V}}, X_{\alpha_s^s}, X_{\beta_s^s}).
\end{align*}
After considering all the pairs $(R_s^2, C_s^2)$, $s=2,3,\ldots, k$, and applying Lemma \ref{nvequ} repeatedly to all the quadruples above, the integral in
\eqref{thm2equ12} is non-vanishing if and only if the following integral is non-vanishing
\begin{equation}\label{thm2equ13}
\int_{[\prod_{j=1}^{k} R_j^2 \wt{V}]} \varphi(rvg) \psi_{\wt{V}}^{-1}(v) dr dv,
\end{equation}
which is exactly of the kind of integrals in \eqref{thm2equ1}.

Next, we start with $\psi_{\ul{p}, \ul{a}}$-Fourier coefficients attached to the symplectic partition $[(2k+1)^2p_1^{e_1}p_2^{e_2}\cdots p_r^{e_r}]$.
By Lemma \ref{halfheisen},
we add $Y$, as defined in \eqref{equy} corresponding to
the partition $[(2k+1)^2p_1^{e_1}p_2^{e_2}\cdots p_r^{e_r}]$,
to the integration domain of the Fourier coefficients,
without changing the non-vanishing property.
Then, after conjugating by the Weyl element
$\diag(I_{2k+1}, \omega_1, I_{2k+1})$,
we see that a Fourier coefficient of $\varphi$ attached to
the partition $[(2k+1)^2p_1^{e_1}p_2^{e_2}\cdots p_r^{e_r}]$
is non-vanishing if and only if the following
integral is non-vanishing
\begin{equation}\label{thm2equ7}
\int_{[W]} \varphi(wg) \psi_W^{-1}(w) dw,
\end{equation}
where elements in $W$ have the following form
\begin{equation}\label{thm2equ8}
\begin{pmatrix}
z & q_1 & q_2\\
0 & v' & q_1^*\\
0 & 0 & z^*
\end{pmatrix}
\begin{pmatrix}
I_{2k+1} & 0 & 0\\
p_1 & I_{2n-4k-2} & 0\\
p_2 & p_1^* & I_{2k+1}
\end{pmatrix},
\end{equation}
where $z\in N_{2k+1}$, and also $q_1 \in \Mat_{(2k+1) \times (2n-4k-2)}$
with certain conditions;
$p_2, q_2 \in \Mat'_{(2k+1) \times (2k+1)}$ with
$p_2(i,j)=q_2(i,j)=0$ for any $i \geq j$;
$p_1 \in \Mat_{(2n-4k-2)\times (2k+1)}$
with certain conditions; and finally
$v' \in \omega_1 \wt{Y}V_{[p_1^{e_1}p_2^{e_2}\cdots p_r^{e_r}],2} \omega_1^{-1}$.
The character is given by
$$
\psi_W(\diag(z, v', z^*))
=  \psi(\sum_{i=1}^{2k} z_{i,i+1})\cdot
\psi_{[p_1^{e_1}p_2^{e_2}\cdots p_r^{e_r}],\ul{a}}(\omega_1^{-1} v' \omega_1).
$$

Now, we specify the conditions on $q_1$ and $p_1$ above.
For $q_1$, the last row is zero, assume that $1 \leq i_0 \leq k$ is the
first row that $q_1$ has zero entries. For $i_0 \leq i \leq k$,
assume that $\alpha_{i,1}, \alpha_{i,2}, \ldots, \alpha_{i,s_i}$ are
all the roots such that $q_1$ has zero entries at corresponding places
from right to the left with $s_i \in \BZ_{>0}$. Assume that $k+1 \leq i_{00} \leq 2k$ is the last row that $q_1$ has nonzero entries. For $k+1 \leq i \leq i_{00}$, assume that $\alpha_{i,1}, \alpha_{i,2}, \ldots, \alpha_{i,s_i}$ are
all the roots such that $q_1$ has nonzero entries at corresponding places
from left to the right with $s_i \in \BZ_{>0}$.
Then
\begin{itemize}
\item $i_0+1$ is the first column that $p_1$ has
non-zero entries;
\item for $i_0 \leq i \leq k$, in the column $i+1$,
there are exactly $s_i$ roots $\beta_{i,j}$, $1 \leq j \leq s_i$,
such that $p_1$ has non-zero entries at corresponding places,
from bottom to top; and
\item for $1 \leq j \leq s_i$, $\alpha_{i,j} + \beta_{i,j} = e_i - e_{i+1}$;
\item $i_{00}+1$ is the last column that $p_1$ has zero entries;
\item for $k+1 \leq i \leq i_{00}$, in the column $i+1$,
there are exactly $t_i$ roots $\beta_{i,j}$, $1 \leq j \leq s_i$,
such that $p_1$ has zero entries at corresponding places,
from top to bottom; and
\item for $1 \leq j \leq s_i$, $\alpha_{i,j} + \beta_{i,j} = e_i - e_{i+1}$.
\end{itemize}
For $i_0 \leq i \leq i_{00}$, define
$R_i^1 = \prod_{j=1}^{s_i} X_{\alpha_{i,j}}$, $C_i^1 = \prod_{j=1}^{s_i} X_{\beta_{i,j}}$.

As above, for $1 \leq i \leq k$, we define $R^2_i = \prod_{j=1}^i X_{\alpha^i_j}$,
with $\alpha^i_j = e_i+e_{2k+1-i+j}$,
and $C^2_i = \prod_{j=1}^i X_{\beta^i_j}$,
with $\beta^i_j = -e_{2k+1-i+j}-e_{i+1}$.

Write $W=\prod_{i=1}^{k} C_i^2 \prod_{i=i_0}^{k} C_i^1 \prod_{i=k+1}^{i_{00}} R_i^1 \wt{W}$,
with
$$\prod_{i=1}^{k} C_i^2 \prod_{i=i_0}^{k} C_i^1\prod_{i=k+1}^{i_{00}} R_i^1 \cap \wt{W}=\{1\},$$ and
write $\psi_{\wt{W}} = \psi_{W}|_{\wt{W}}$.

We are ready to apply Lemma \ref{nvequ} to the integration over
$$\prod_{i=1}^{k} C_i^2 \prod_{i=i_0}^{k} C_i^1 \prod_{i=k+1}^{i_{00}} R_i^1.$$
For this, we need to consider pairs of groups ordered as following
\begin{align*}
&(R_1^2, C_1^2), (R_2^2, C_2^2), \ldots, (R_{i_0-1}^2, C_{i_0-1}^2);\\
& (R_{i_0}^2 C_{i_0}^1, C_{i_0}^2 C_{i_0}^1), (R_{i_0+1}^2 C_{i_0+1}^1, C_{i_0+1}^2 C_{i_0+1}^1), \ldots, (R_{k}^2 C_{k}^1, C_{k}^2 C_{k}^1);\\
&(C_{i_{00}}^1, R_{i_{00}}^1),(C_{i_{00}-1}^1, R_{i_{00}-1}^1), \cdots (C_{k+1}^1, R_{k+1}^1).
\end{align*}

First, for the pair $(R_1^2, C_1^2)$, we consider the following quadruple
$$(\prod_{i=2}^{k} C_i^2 \prod_{i=i_0}^{k} C_i^1 \prod_{i=k+1}^{i_{00}} R_i^1 \wt{W}, \psi_{\wt{W}},
X_{\alpha_1^1}, X_{\beta_1^1}).$$
By Lemma \ref{nvequ}, we can see that the integral in
\eqref{thm2equ7} is non-vanishing if and only if the following integral is
non-vanishing
\begin{equation}\label{thm2equ14}
\int_{[R_1^2 \prod_{i=2}^{k} C_i^2 \prod_{i=i_0}^{k} C_i^1 \prod_{i=k+1}^{i_{00}} R_i^1 \wt{W}]} \varphi(xwg) \psi_{\wt{W}}^{-1}(w) dxdw.
\end{equation}

For the pair $(R_t^2, C_t^2)$, $t=2, 3, \ldots, i_0-1$, we consider the following sequence of quadruples
\begin{align*}
&(\prod_{j=2}^t X_{\beta_j^t} \prod_{i=1}^{t-1} R_i^2 \prod_{i=t+1}^{k} C_i^2 \prod_{i=i_0}^{k} C_i^1 \prod_{i=k+1}^{i_{00}} R_i^1 \wt{W}, \psi_{\wt{W}},
X_{\alpha_1^t}, X_{\beta_1^t}),\\
&(X_{\alpha_1^t}\prod_{j=3}^t X_{\beta_j^t} \prod_{i=1}^{t-1} R_i^2 \prod_{i=t+1}^{k} C_i^2 \prod_{i=i_0}^{k} C_i^1 \prod_{i=k+1}^{i_{00}} R_i^1 \wt{W}, \psi_{\wt{W}},
X_{\alpha_2^t}, X_{\beta_2^t}),\\
& \cdots \\
& (\prod_{j=1}^{t-1} X_{\alpha_j^t} \prod_{i=1}^{t-1} R_i^2 \prod_{i=t+1}^{k} C_i^2 \prod_{i=i_0}^{k} C_i^1 \prod_{i=k+1}^{i_{00}} R_i^1 \wt{W}, \psi_{\wt{W}},
X_{\alpha_t^t}, X_{\beta_t^t}).
\end{align*}
After considering all the pairs of groups $(R_t^2, C_t^2)$, $t=2, 3, \ldots, i_0-1$, and applying Lemma \ref{nvequ} repeatedly to all the above quadruples, we can see that the integral in
\eqref{thm2equ14} is non-vanishing if and only if the following integral is
non-vanishing
\begin{equation}\label{thm2equ15}
\int_{[\prod_{i=1}^{i_0-1} R_i^2 \prod_{i=i_0}^{k} C_i^2 C_i^1 \prod_{i=k+1}^{i_{00}} R_i^1 \wt{W}]} \varphi(xwg) \psi_{\wt{W}}^{-1}(w) dxdw.
\end{equation}

For the pair $(R_{t}^2 C_{t}^1, C_{t}^2 C_{t}^1)$, $t=i_0, i_0+1, \ldots, k$,
we consider the following sequence of quadruples
\begin{align*}
&(\prod_{j=2}^{t} X_{\beta_j^t} \prod_{i=1}^{t-1} R_i^2
\prod_{i=i_0}^{t-1} R_i^1
C_t^1 \prod_{i=t+1}^{k} C_i^2C_i^1  \prod_{i=k+1}^{i_{00}} R_i^1 \wt{W}, \psi_{\wt{W}},
X_{\alpha_1^t}, X_{\beta_1^t}),\\
& (X_{\alpha_1^t} \prod_{j=3}^{t} X_{\beta_j^t} \prod_{i=1}^{t-1} R_i^2
\prod_{i=i_0}^{t-1} R_i^1
C_t^1 \prod_{i=t+1}^{k} C_i^2C_i^1  \prod_{i=k+1}^{i_{00}} R_i^1 \wt{W}, \psi_{\wt{W}},
X_{\alpha_2^t}, X_{\beta_2^t}),\\
& \cdots,\\
& (\prod_{j=1}^{t-1} X_{\alpha_j^t} \prod_{i=1}^{t-1} R_i^2
\prod_{i=i_0}^{t-1} R_i^1
C_t^1 \prod_{i=t+1}^{k} C_i^2C_i^1  \prod_{i=k+1}^{i_{00}} R_i^1 \wt{W}, \psi_{\wt{W}},
X_{\alpha_t^t}, X_{\beta_t^t});\\
& (\prod_{j=2}^{s_t} X_{\beta_{t,j}} \prod_{i=1}^{t} R_i^2
\prod_{i=i_0}^{t-1} R_i^1
\prod_{i=t+1}^{k} C_i^2C_i^1  \prod_{i=k+1}^{i_{00}} R_i^1 \wt{W}, \psi_{\wt{W}},
X_{\alpha_{t,1}}, X_{\beta_{t,1}}),\\
& (X_{\alpha_{t,1}} \prod_{j=3}^{s_t} X_{\beta_{t,j}} \prod_{i=1}^{t} R_i^2
\prod_{i=i_0}^{t-1} R_i^1
\prod_{i=t+1}^{k} C_i^2C_i^1  \prod_{i=k+1}^{i_{00}} R_i^1 \wt{W}, \psi_{\wt{W}},
X_{\alpha_{t,2}}, X_{\beta_{t,2}}),\\
& \cdots,\\
& (\prod_{j=1}^{s_t-1} X_{\alpha_{t,j}} \prod_{i=1}^{t} R_i^2
\prod_{i=i_0}^{t-1} R_i^1
\prod_{i=t+1}^{k} C_i^2C_i^1  \prod_{i=k+1}^{i_{00}} R_i^1 \wt{W}, \psi_{\wt{W}},
X_{\alpha_{t,s_t}}, X_{\beta_{t,s_t}}).
\end{align*}
After considering all the pairs of groups $(R_{t}^2 C_{t}^1, C_{t}^2 C_{t}^1)$, $t=i_0, i_0+1, \ldots, k$, and applying Lemma \ref{nvequ} repeatedly to all the above quadruples, we can see that the integral in
\eqref{thm2equ15} is non-vanishing if and only if the following integral is
non-vanishing
\begin{equation}\label{thm2equ16}
\int_{[\prod_{i=1}^{k} R_i^2
\prod_{i=i_0}^{k} R_i^1
 \prod_{i=k+1}^{i_{00}} R_i^1 \wt{W}]} \varphi(xwg) \psi_{\wt{W}}^{-1}(w) dxdw.
\end{equation}

For the pair $(C_{t}^1, R_{t}^1)$, $t=i_{00}, i_{00}-1, \ldots, k+1$, we consider the following sequence of quadruples
\begin{align*}
& (\prod_{j=2}^{s_t} X_{\alpha_{t,j}}
\prod_{i=1}^{k} R_i^2
\prod_{i=i_0}^{k} R_i^1
\prod_{i=t+1}^{i_{00}} C_i^1
 \prod_{i=k+1}^{t-1} R_i^1 \wt{W},\psi_{\wt{W}}, X_{\beta_{t,1}}, X_{\alpha_{t,1}}),\\
& (X_{\beta_{t,1}} \prod_{j=3}^{s_t} X_{\alpha_{t,j}}
\prod_{i=1}^{k} R_i^2
\prod_{i=i_0}^{k} R_i^1
\prod_{i=t+1}^{i_{00}} C_i^1
 \prod_{i=k+1}^{t-1} R_i^1 \wt{W},\psi_{\wt{W}}, X_{\beta_{t,2}}, X_{\alpha_{t,2}}),\\
& \cdots,
\end{align*}
\begin{align*}
& (\prod_{j=1}^{s_t-1} X_{\beta_{t,j}}
\prod_{i=1}^{k} R_i^2
\prod_{i=i_0}^{k} R_i^1
\prod_{i=t+1}^{i_{00}} C_i^1
 \prod_{i=k+1}^{t-1} R_i^1 \wt{W},\psi_{\wt{W}}, X_{\beta_{t,s_t}}, X_{\alpha_{t,s_t}}).
\end{align*}
After considering all the pairs of groups $(C_{t}^1, R_{t}^1)$, $t=i_{00}, i_{00}-1, \ldots, k+1$, and applying Lemma \ref{nvequ} repeatedly to all the above quadruples, we can see that the integral in
\eqref{thm2equ16} is non-vanishing if and only if the following integral is
non-vanishing
\begin{equation}\label{thm2equ17}
\int_{[\prod_{i=1}^{k} R_i^2
\prod_{i=i_0}^{k} R_i^1
\prod_{i=k+1}^{i_{00}} C_i^1
\wt{W}]} \varphi(xwg) \psi_{\wt{W}}^{-1}(w) dxdw,
\end{equation}
which is exactly of the kind of integrals in \eqref{thm2equ1}.
This completes the proof of the proposition.
\end{proof}


\section{Special Unipotent Orbits and Fourier Coefficients}

A symplectic partition is called {\it special}
if it has an even number of even parts between any two consecutive
odd ones and an even number of even parts greater than the largest odd part,
equivalently, the number of even parts (counted with multiplicity),
greater than any odd part is even.

It is a general expectation that for any irreducible automorphic representation $\pi$, the partitions in
the set $\mathfrak{n}^m(\pi)$ should be all special. This has been checked in \cite[Theorem 2.1]{GRS03} for $\Sp_{2n}$, and
a very sketching argument is given in \cite{G06} for all classical groups.
We give here a complete proof for the fact that for any irreducible
automorphic representation $\pi$ of symplectic groups $G_n(\BA)$, a partition $\ul{p}$
providing non-vanishing Fourier coefficients implies that
the symplectic expansion $\ul{p}^G$ of $\ul{p}$ also provides non-vanishing Fourier coefficients.
As an immediate corollary, this implies that the partitions in
the set $\mathfrak{n}^m(\pi)$ are all special. We remark that the argument given here has potential applications
to other problems on Fourier coefficients of automorphic representations, which will be dealt with in our future work.

For a symplectic partition $\ul{p}$ of $2n$,
which is not special, the smallest
special symplectic partition which is greater than $\ul{p}$
is call the $G$-expansion of $\ul{p}$, denoted
by $\ul{p}^G$, with $G=\Sp_{2n}$ here. Theorem 6.3.9 of \cite{CM93} gives a recipe for passing from
a symplectic partition $\ul{p}$ to its $G$-expansion, which can be
explicitly described as follows.

Given a symplectic partition $\ul{p}$ of $2n$, whose odd parts occur with even
multiplicity, we may write $\ul{p}=[p_1 p_2 \cdots p_r]$
with $p_1 \geq p_2 \geq \cdots \geq p_r >0$. Enumerate the
indices $i$ such that $p_{2i} = p_{2i+1}$ is odd and
$p_{2i-1} \neq p_{2i}$ as $i_1 < \cdots < i_t$. Then
the $G$-expansion of $\ul{p}$ can be obtained by
replacing each pair of parts $(p_{2i_j}, p_{2i_j+1})$
by $(p_{2i_j}+1, p_{2i_j+1}-1)$, respectively,
and leaving the other parts alone.
For example, for the symplectic partition $\ul{p}=[65^243^221^2]$,
which is not special, we have $p_1 \neq p_2 = p_3 =5$,
and $p_7 \neq p_8 = p_9 =1$. Then $\ul{p}^G = [6^2 4^2 3^22^2]$,
which is exactly obtained by replacing the pair $(5, 5)$
by $(6,4)$, $(1,1)$ by $(2,0)$, and leaving the other parts alone.

\begin{thm}\label{thm9}
Let $\pi$ be an irreducible automorphic representation
of $G_n(\BA)=\Sp_{2n}(\BA)$. If $\pi$ has a nonzero Fourier coefficient
attached to a non-special symplectic partition $\ul{p}$ of $2n$,
then $\pi$ must have a nonzero Fourier coefficient
attached to $\ul{p}^G$, the $G$-expansion of $\ul{p}$.
\end{thm}

\begin{proof}
We are going to prove this theorem by induction on the rank of $G_n$,
and reduce the proof to Lemma \ref{lem3}.
When $n=1$, it is obvious that the statement is true, since any symplectic partition of $2$ is special. We assume that the statement is true for any irreducible automorphic representation of ${\Sp}_{2m}(\BA)$ with $m <n$.

Now we consider the case of ${\Sp}_{2n}(\BA)$.
Since $\ul{p}$ is not special, we may assume that
$(2n_2+1)$ is the largest odd part of $\ul{p}$ such that
the number of even parts (counted with multiplicity)
greater than $(2n_2+1)$ is odd. We may write
$\ul{p}$ as
$$[p_1 p_2 \cdots p_r (2n_1)(2n_2+1)^{2k} p_{r+1} \cdots p_{r+s}],$$
where $2n_2+1 > p_{r+1}$ if $s>0$, and the partition $[p_1 p_2 \cdots p_r]$ is special.

Since by the above assumption, the number of even parts greater than or equal to $2n_1$ is even,
after taking the descent of $\pi$ corresponding to the partition $[p_1 p_2 \cdots p_r 1^{2n-\sum_{i=1}^r p_i}]$,
we will also get an automorphic representation of $\Sp_{2n-\sum_{i=1}^r p_i}(\BA)$.
Note that the notion of descent with respect to any symplectic partition is defined in Section 1 of \cite{GRS03}.
Therefore, applying Lemma \ref{lem1}
and Proposition \ref{prop2} repeatedly, it is enough to consider $\ul{p}$
of the following form
\begin{equation}\label{thm9equ1}
\ul{p}=[(2n_1)(2n_2+1)^{2k} p_{1} \cdots p_{s}],
\end{equation}
where $2n_2+1 > p_{1}$ if $s>0$.

By Lemma \ref{lem3} below, $\pi$ has a nonzero Fourier coefficient
attached to the partition
\begin{equation}\label{thm9equ2}
[(2n_1)(2n_2+2)(2n_2+1)^{2k-2}(2n_2)p_{1} \cdots p_{s}].
\end{equation}

Applying Lemma \ref{lem1} and Proposition \ref{prop2} repeatedly,
$\pi$ has a nonzero Fourier coefficient
attached to the partition in \eqref{thm9equ2} if and only if
it has a nonzero Fourier coefficient attached to the
following partition
\begin{equation}\label{thm9equ3}
[(2n_1)(2n_2+2)(2n_2+1)^{2k-2}] \circ [(2n_2)p_{1} \cdots p_{s}].
\end{equation}

By Lemma \ref{lem1} and Proposition \ref{prop2},
it remains to prove that for any irreducible automorphic
representation $\pi'$ of $G_m(\BA)$, with $2m=2n-2n_1-(2n_2+2)-(2n_2+1)(2k-2)$,
if it has a nonzero Fourier coefficient attached to
the partition $[(2n_2)p_{1} \cdots p_{s}]$ and if this partition
is not special, then $\pi'$ must have a nonzero
Fourier coefficient attached to
$[(2n_2)p_{1} \cdots p_{s}]^G$, the $G$-expansion of
$[(2n_2)p_{1} \cdots p_{s}]$. Since $m < n$,
this is true by induction.

Therefore, $\pi$ has a nonzero Fourier coefficient attached to the
following partition
\begin{equation}\label{thm9equ4}
[(2n_1)(2n_2+2)(2n_2+1)^{2k-2}] \circ [(2n_2)p_{1} \cdots p_{s}]^G.
\end{equation}
Applying Lemma \ref{lem1} and Proposition \ref{prop2} repeatedly again,
$\pi$ has a nonzero Fourier coefficient attached to the
following partition
\begin{equation}\label{thm9equ5}
[(2n_1)(2n_2+2)(2n_2+1)^{2k-2}[(2n_2)p_{1} \cdots p_{s}]^G],
\end{equation}
which is exactly $[(2n_1)(2n_2+1)^{2k} p_{1} \cdots p_{s}]^G$, the $G$-expansion
of $\ul{p}$.

This completes the proof of the theorem, up to Lemma \ref{lem3}.
\end{proof}

Theorem \ref{thm9} easily implies the following corollary.

\begin{cor}[Theorem 2.1 of \cite{GRS03}]\label{cor1}
Let $\pi$ be an irreducible automorphic representation
of $\Sp_{2n}(\BA)$. Any partition in the set $\mathfrak{n}^m(\pi)$ is special.
\end{cor}

It remains to prove

\begin{lem}\label{lem3}
Let $\pi$ be an irreducible automorphic representation
of $\Sp_{2n}(\BA)$. Assume that $\pi$ has a nonzero Fourier coefficient
attached to a symplectic partition $\ul{p}=[(2n_1)(2n_2+1)^{2k}p_{1}^{e_1} \cdots p_{s}^{e_s}]$ of $2n$ with the conditions that
\begin{enumerate}
\item[(1)] $k \geq 1$ and $2n_1 \geq 2n_2 +2$; and
\item[(2)] if $s>0$, then $2n_2+1 > p_{1}$; and $e_i = 1$ if $p_i$ is even, and $e_i=2$ if $p_i$ is odd, for $1 \leq i \leq s$.
\end{enumerate}
Then $\pi$ must have a nonzero Fourier coefficient
attached to the following partition
$$
[(2n_1)(2n_2+2)(2n_2+1)^{2k-2}(2n_2)p_{1}^{e_1} \cdots p_{s}^{e_s}].
$$
\end{lem}

\begin{proof}
First, we note that if $s=0$ and $n_2=0$, then
$\ul{p}=[(2n_1)1^{2k}]$.
By \cite[Lemma 1]{GRS03}, a Fourier coefficient of $\pi$
with respect to $\ul{p}=[(2n_1)1^{2k}]$ and a square class $\alpha$ is non-vanishing if and only if the automorphic descent (see \cite[(1.3)]{GRS03} and \cite{GRS11}) of $\pi$ with respect to $\ul{p}=[(2n_1)1^{2k}]$ and the square class $\alpha$ is non-vanishing.
After taking this descent operation, we get a genuine automorphic representation
of $\tilsp_{2k}(\BA)$, which has a nonzero Fourier coefficient
attached to $[21^{2k-2}]$.
Hence, in this case, by Lemma \ref{lem1}, it is obvious that
$\pi$ have a nonzero nontrivial Fourier coefficient
attached to $[(2n_1)21^{2k-2}]$.
Therefore, we may assume that if $s=0$, then $n_2 \geq 1$.

Take $\varphi \in \pi$. We start from a nonzero $\psi_{\ul{p}, \{a_1\} \cup \ul{a}}$-Fourier
coefficient of $\varphi$ attached to $\ul{p}$,
where $\ul{a}= \{a_i\in F^*/(F^*)^2\ |\ e_i =1, 1 \leq i \leq s\}$ and $a_1\in F^*/(F^*)^2$.

By \eqref{fc}, $\varphi^{\psi_{\ul{p}, \{a_1\} \cup \ul{a}}}$
is defined as follows:
\begin{equation}\label{lem3equ1}
\int_{[V_{\ul{p}, 2}]} \varphi(vg) \psi_{\ul{p}, \{a_1\} \cup \ul{a}}^{-1}(v) dv.
\end{equation}

By Corollary \ref{halfheisen} (see also Lemma 1.1 of \cite{GRS03}),
the $\psi_{\ul{p}, \{a_1\} \cup \ul{a}}$-Fourier
coefficient of $\varphi$ in \eqref{lem3equ1} is non-vanishing if and only if
the following integral is non-vanishing:
\begin{equation}\label{lem3equ2}
\int_{[YV_{\ul{p}, 2}]} \varphi(vyg) \psi_{\ul{p}, \{a_1\} \cup \ul{a}}^{-1}(v) dvdy,
\end{equation}
where $Y$ is a group defined in \eqref{equy}, corresponding to
the partition $\ul{p}$.

Let $m=n-n_1-k(2n_2+1)$, then $[p_{1}^{e_1} \cdots p_{s}^{e_s}]$
is a partition of $2m$.
Let $\omega_1$ be a Weyl element of $G_{m}$ which sends the one-dimensional
toric subgroup $\CH_{[p_1^{e_1}\cdots p_s^{e_s}]}$ defined in \eqref{equtorus}
to the following toric subgroup:
\begin{equation*}
\{\diag(t^{p_1-1}, \ldots, t^{1-p_1})\},
\end{equation*}
where the exponents of $t$ are of non-increasing order.

Note that $[(2n_1)(2n_2+1)^{2k}]$ is a partition of $2n-2m$.
Let $\omega_2$ be a Weyl element of $G_{n-m}$ which sends the one-dimensional
toric subgroup $\CH_{[(2n_1)(2n_2+1)^{2k}]}$ defined in \eqref{equtorus}
to the following toric subgroup:
\begin{equation*}
\{\diag(t^{2n_1-1}, \ldots, t^{1-2n_1})\},
\end{equation*}
where the exponents of $t$ are of non-increasing order.
Write $\omega_2$ as
$\begin{pmatrix}
\omega_2^1 & \omega_2^2\\
\omega_2^3 & \omega_2^4
\end{pmatrix}$,
where $\omega_2^i$ is of rank $(n-m) \times (n-m)$, for any $1 \leq i \leq 4$.

Let $\omega_3 = \begin{pmatrix}
\omega_2^1 &0 & \omega_2^2\\
0& \omega_1 &0 \\
\omega_2^3 & 0& \omega_2^4
\end{pmatrix}$.
Conjugating by $\omega_3$, the integral in
\eqref{lem3equ2} becomes:
\begin{equation}\label{lem3equ3}
\int_{[W_1]} \varphi(w \omega_3 g) \psi_{W_1}^{-1}(w) dw,
\end{equation}
where $W_1= \omega_3 YV_{\ul{p}, 2} \omega_3^{-1}$,
$\psi_{W_1}(w)=\psi_{\ul{p}, \{a_1\} \cup \ul{a}}(\omega_3^{-1} w \omega_3)$.
Elements in $W_1$ have the following form
\begin{equation*}
w=\begin{pmatrix}
z & q_1 & q_2\\
0 & v' & q_1^*\\
0 & 0 & z^*
\end{pmatrix}
\begin{pmatrix}
I_{n-m} & 0 & 0\\
p_1 & I_{2m} & 0\\
0 & p_1^* & I_{n-m}
\end{pmatrix},
\end{equation*}
where $z$ is in a subgroup of $N_{n_1+k(2n_2+1)}$; $q_1 \in \Mat_{(n_1+k(2n_2+1)) \times (2m)}$
with certain conditions;
$p_1 \in \Mat_{(2m) \times (n_1+k(2n_2+1))}$
with certain conditions;
$q_2 \in \Mat'_{(n_1+k(2n_2+1)) \times (n_1+k(2n_2+1))}$
with $q_2(i,j) = 0$ for $n_1+2kn_2+1 \leq i \leq n_1+k(2n_2+1)$ and
$1 \leq j \leq k$; and finally
$v' \in \omega_1 \wt{Y}V_{[p_1^{e_1}\cdots p_r^{e_r}],2} \omega_1^{-1}$
with $\wt{Y}$ defined in \eqref{equy}
corresponding to the partition $[p_1^{e_1}\cdots p_s^{e_s}]$. Note that
$\Mat'_{(n_1+k(2n_2+1)) \times (n_1+k(2n_2+1))}$
consists of matrices $q$ in $\Mat_{(n_1+k(2n_2+1)) \times (n_1+k(2n_2+1))}$ satisfying the property
that $q^t v_{n_1+k(2n_2+1)} - v_{n_1+k(2n_2+1)} q=0$, where $v_{n_1+k(2n_2+1)}$ is a matrix only with
ones on the second diagonal. After conjugating by a suitable toric element to adjust the signs, the character can be given by
\begin{align}\label{lem3equ10}
\begin{split}
&\psi_{W_1}(w)\\
=  \ & \psi(\sum_{i=1}^{n_1-n_2-1} w_{i,i+1})
\psi(\sum_{j=n_1-n_2}^{n_1+2kn_2-k-1} w_{j,j+2k+1})\\
\cdot & \psi(\sum_{l=n_1+2kn_2-k}^{n_1+2kn_2-1} w_{l,l+2k+1+2m})
\psi(a_1 w_{n_1+2kn_2,n_1+2kn_2+2k+1+2m})\\
\cdot & \psi_{[p_1^{e_1}\cdots p_s^{e_s}],\ul{a}}(\omega_1^{-1} v' \omega_1).
\end{split}
\end{align}

By Theorem 5.8 of \cite{N99}, as $\varphi$ runs through $\pi$,
the integrals in (1.3) of \cite{GRS03} with respect to the nilpotent orbit corresponding to $(\ul{p}, \{a_1\} \cup \ul{a})$ define a geniune representation
of $\wt{\Sp}_{2k}(\BA)$. Let $\alpha$ be the highest root of
${\Sp}_{2k}$ and $X_{\alpha}$ be the corresponding
one-dimensional root subgroup. After taking the
Fourier expansion the integral in (1.3) of \cite{GRS03} along $X_{\alpha}(\BA)$, we get a nonzero
nontrivial Fourier coefficient.
By the property at the end of Section 1 of \cite{GRS03}, the integral in (1.3) of \cite{GRS03} has a nonzero nontrivial Fourier coefficient along $X_{\alpha}(\BA)$ if and only if the integral in \eqref{lem3equ2} has a nonzero nontrivial Fourier coefficient along $X_{\alpha}(\BA)$.
Therefore, the integral \eqref{lem3equ3} has a nonzero nontrivial Fourier coefficient along $X_{\alpha}(\BA)$.
Without loss of generality,
we assume that the following Fourier coefficient is non-vanishing
\begin{equation}\label{lem3equ6}
\int_{[X_{\alpha}W_1]} \varphi(xw \omega_3 g) \psi_{W_1}^{-1}(w)\psi^{-1}(\gamma x) dxdw,
\end{equation}
with $X_{\alpha}$ being diagonally embedded into $\Sp_{2n}$.

We define that
$\alpha_i = e_{n_1-n_2+(n_2-i)(2k+1)}-e_{n_1-n_2+(n_2-i)(2k+1)+1}$ and
$\beta_i = e_{n_1-n_2+(n_2-i)(2k+1)+1}-e_{n_1-n_2+(n_2-i)(2k+1)+2k+1}$
for $1 \leq i \leq n_2$,
and let $X_{\alpha_i}$ and $X_{\beta_i}$ be the corresponding one-dimensional
root subgroups, respectively.
Write $X_{\alpha}W_1 = (\prod_{i=1}^{n_2} X_{\alpha_i}) \wt{W}_1$,
with $\prod_{i=1}^{n_2} X_{\alpha_i} \cap \wt{W}_1=\{1\}$,
and denote $\psi_{\wt{W}_1} = \psi_{X_{\alpha}W_1}|_{\wt{W}_1}$,
where $\psi_{X_{\alpha}W_1}(xw) = \psi_{W_1}(w)\psi(\gamma x)$.
Consider the quadruple
$(\prod_{i=2}^{n_2} X_{\alpha_i} \wt{W}_1, \psi_{\wt{W}_1}, X_{\alpha_1}, X_{\beta_1})$.
It is easy to see that it satisfies the conditions for
Lemma \ref{nvequ}. Therefore, by Lemma \ref{nvequ},
the integral in \eqref{lem3equ3} is non-vanishing
if and only if the following integral is non-vanishing
\begin{equation}\label{lem3equ4}
\int_{[X_{\beta_1}\prod_{i=2}^{n_2} X_{\alpha_i} \wt{W}_1]} \varphi(xw \omega_3 g) \psi_{\wt{W}_1}^{-1}(w) dxdw.
\end{equation}

Then we consider the following sequence of quadruples
\begin{align*}
&(X_{\beta_1}\prod_{i=3}^{n_2} X_{\alpha_i} \wt{W}_1, \psi_{\wt{W}_1}, X_{\alpha_2}, X_{\beta_2}),\\
&(\prod_{j=1}^{2} X_{\beta_j} \prod_{i=4}^{n_2} X_{\alpha_i} \wt{W}_1, \psi_{\wt{W}_1}, X_{\alpha_3}, X_{\beta_3}),\\
&\cdots\\
&(\prod_{j=1}^{n_2-1} X_{\beta_j} \wt{W}_1, \psi_{\wt{W}_1}, X_{\alpha_{n_2}},
X_{\beta_{n_2}}).
\end{align*}
Applying Lemma \ref{nvequ} repeatedly, the integral
in \eqref{lem3equ4} is non-vanishing if and only if the
following integral is non-vanishing
\begin{equation}\label{lem3equ5}
\int_{[\prod_{j=1}^{n_2} X_{\beta_j} \wt{W}_1]} \varphi(xw \omega_3 g) \psi_{\wt{W}_1}^{-1}(w) dxdw.
\end{equation}

Let $W_2=\prod_{j=1}^{n_2} X_{\beta_j} \wt{W}_1$, and
$\psi_{W_2}(xw) = \psi_{\wt{W}_1}(w)$, for $x \in \prod_{j=1}^{n_2} X_{\beta_i}$,
and $w \in \wt{W}_1$.
For $1 \leq i \leq n_2$ and $1 \leq j \leq 2k-2$,
let
\begin{align*}
\alpha_{i,j} & = e_{n_1-n_2+1+(i-1)(2k+1)+j} - e_{n_1-n_2+1+i(2k+1)},\\
\beta_{i,j}& =e_{n_1-n_2+1+(i-1)(2k+1)}-e_{n_1-n_2+1+(i-1)(2k+1)+j},
\end{align*}
and let $X_{\alpha_{i,j}}$, $X_{\beta_{i,j}}$ be the corresponding one-dimensional
root subgroups.
Let $X_i = \prod_{j=1}^{2k-2} X_{\alpha_{i,j}}$,
$Y_i = \prod_{j=1}^{2k-2} X_{\beta_{i,j}}$.
Write $W_2 = (\prod_{i=1}^{n_2} X_i) \wt{W}_2$,
with $\prod_{i=1}^{n_2} X_i \cap \wt{W}_2=\{1\},$
and denote $\psi_{\wt{W}_2} = \psi_{W_2}|_{\wt{W}_2}$.

Consider the following sequence of quadruples
\begin{align*}
&(\prod_{i=2}^{n_2} X_{i} \wt{W}_2, \psi_{\wt{W}_2}, X_{1}, Y_{1}),\\
&(Y_1 \prod_{i=3}^{n_2} X_{i} \wt{W}_2, \psi_{\wt{W}_2}, X_{2}, Y_2),\\
&\cdots\\
&(\prod_{j=1}^{n_2-1} Y_j \wt{W}_2, \psi_{\wt{W}_2}, X_{{n_2}}, Y_{n_2}).
\end{align*}
Applying Lemma \ref{nvequ} repeatedly, the integral
in \eqref{lem3equ5} is non-vanishing if and only if the
following integral is non-vanishing
\begin{equation}\label{lem3equ7}
\int_{[\prod_{j=1}^{n_2} Y_j \wt{W}_2]} \varphi(xw \omega_3 g) \psi_{\wt{W}_2}^{-1}(w) dxdw.
\end{equation}

Define $W_3 = \prod_{j=1}^{n_2} Y_j \wt{W}_2$.
Let $\psi_{W_3}(yw) = \psi_{\wt{W}_2}(w)$
For $1 \leq i \leq n_2-1$, define
\begin{eqnarray*}
\gamma_i &=& e_{n_1-n_2-1+i(2k+1)}-e_{n_1-n_2+1+i(2k+1)}\\
\delta_i &=& e_{n_1-n_2+1+i(2k+1)}-e_{n_1-n_2+i(2k+1)+2k}.
\end{eqnarray*}
For $i=n_2$, define
\begin{eqnarray*}
\gamma_{n_2}& =& e_{n_1-n_2-1+n_2(2k+1)}-e_{n_1-n_2+1+n_2(2k+1)}\\
\delta_{n_2} &=& e_{n_1-n_2+1+n_2(2k+1)}+e_{n_1-n_2+1+n_2(2k+1)}.
\end{eqnarray*}
For $1 \leq i \leq n_2$, let
$X_{\gamma_i}$, $X_{\delta_i}$ be the corresponding
one-dimensional root subgroups.
Write $W_3 = \prod_{i=1}^{n_2} X_{\gamma_i} \wt{W}_3$
with $\prod_{i=1}^{n_2} X_{\gamma_i} \cap \wt{W}_3 = \{1\}$,
and denote $\psi_{\wt{W}_3} = \psi_{W_3}|_{\wt{W}_3}$.

Consider the following sequence of quadruples
\begin{align*}
&(\prod_{i=2}^{n_2} X_{\gamma_i} \wt{W}_3, \psi_{\wt{W}_3}, X_{\gamma_1}, X_{\delta_1}),\\
&(X_{\delta_1} \prod_{i=3}^{n_2} X_{\gamma_i} \wt{W}_3, \psi_{\wt{W}_3}, X_{\gamma_2}, X_{\delta_2}),\\
&\cdots\\
&(\prod_{j=1}^{n_2-1} X_{\delta_j} \wt{W}_3, \psi_{\wt{W}_3}, X_{\gamma_{n_2}}, X_{\delta_{n_2}}).
\end{align*}
Applying Lemma \ref{nvequ} repeatedly, the integral
in \eqref{lem3equ7} is non-vanishing if and only if the
following integral is non-vanishing
\begin{equation}\label{lem3equ8}
\int_{[\prod_{j=1}^{n_2} X_{\delta_j} \wt{W}_3]} \varphi(xw \omega_3 g) \psi_{\wt{W}_3}^{-1}(w) dxdw.
\end{equation}

Define $W_4 = \prod_{j=1}^{n_2} X_{\delta_j} \wt{W}_3$.
Let $\psi_{W_4}(xw) = \psi_{\wt{W}_3}(w)$
For $1 \leq s \leq k-1$, define
\begin{eqnarray*}
\gamma_s &=& e_{n_1-n_2-1+n_2(2k+1)}-e_{n_1-n_2+1+n_2(2k+1)+s}\\
\delta_s &=& e_{n_1-n_2+1+n_2(2k+1)+s} + e_{n_1-n_2+1+n_2(2k+1)},
\end{eqnarray*}
and $X_{\gamma_s}$, $X_{\delta_s}$ be the corresponding
quadruples.
Let $\ol{X} = \prod_{s=1}^{k-1} X_{\gamma_s}$,
$\ol{Y} = \prod_{s=1}^{k-1} X_{\delta_s}$.
Write $W_4 = \ol{X} \wt{W}_4$ with $\ol{X} \cap \wt{W}_4 = \{I_{2n}\}$,
and denote $\psi_{\wt{W}_4} = \psi_{W_4}|_{\wt{W}_4}$.
Consider the following quadruple
$(\wt{W}_4, \psi_{\wt{W}_4}, \ol{X}, \ol{Y})$.
By Lemma \ref{nvequ}, the integral in \eqref{lem3equ8}
is non-vanishing if and only if the following integral
is non-vanishing
\begin{equation}\label{lem3equ9}
\int_{[\ol{Y} \wt{W}_4]} \varphi(yw \omega_3 g) \psi_{\wt{W}_4}^{-1}(w) dydw.
\end{equation}

Let $W_5 = \ol{Y} \wt{W}_4$, and $\psi_{W_5}(xw) = \psi_{\wt{W}_4}(w)$.
Now we describe the structure of elements in $W_5$.
Elements in $W_5$ have the following form
\begin{equation}\label{lem3equ17}
w=\begin{pmatrix}
z & q_1 & q_2\\
0 & v' & q_1^*\\
0 & 0 & z^*
\end{pmatrix}
\begin{pmatrix}
I_{n-m} & 0 & 0\\
p_1 & I_{2m} & 0\\
0 & p_1^* & I_{n-m}
\end{pmatrix},
\end{equation}
where $z$ runs over the unipotent radical
of a parabolic subgroup
of $\GL_{n_1+k(2n_2+1)}$ with
Levi isomorphic to $\GL_{n_1-n_2+1} \times \GL_{2k+1}^{n_2} \times \GL_{k-1}$,
and $z_{i,j}=0$ for $n_1-n_2-1+n_2(2k+1) \leq i \leq n_1-n_2+1+n_2(2k+1)$ and
$n_1-n_2+2+n_2(2k+1) \leq j \leq n_1+k(2n_2+1)$;
$q_1 \in \Mat_{(n_1+k(2n_2+1)) \times (2m)}$
with certain conditions;
$p_1 \in \Mat_{(2m) \times (n_1+k(2n_2+1))}$
with certain conditions;
$q_2 \in \Mat'_{(n_1+k(2n_2+1)) \times (n_1+k(2n_2+1))}$
with $q_2(i,j) = 0$ for $n_1+2kn_2+2 \leq i \leq n_1+k(2n_2+1)$ and
$1 \leq j \leq k-1$, where
$$\Mat'_{(n_1+k(2n_2+1)) \times (n_1+k(2n_2+1))} \subset \Mat_{(n_1+k(2n_2+1)) \times (n_1+k(2n_2+1))}$$
consists of
matrices $q$
satisfying the property
that
$$q^t v_{n_1+k(2n_2+1)} - v_{n_1+k(2n_2+1)} q=0$$
with $v_{n_1+k(2n_2+1)}$ being a matrix only with
ones on the second diagonal; and finally
$v' \in \omega_1 \wt{Y}V_{[p_2^{e_2}\cdots p_r^{e_r}],2} \omega_1^{-1}$
with $\wt{Y}$ defined in \eqref{equy}
corresponding to the partition $[p_1^{e_1}\cdots p_s^{e_s}]$.
The character is given by
\begin{align}\label{lem3equ11}
\begin{split}
&\psi_{W_5}(w)\\
=  \ & \psi(\sum_{i=1}^{n_1-n_2-1} w_{i,i+1})
\psi(\sum_{j=n_1-n_2}^{n_1+2kn_2-k-1} w_{j,j+2k+1})\psi(\gamma w_{n_1-n_2+1, n_1-n_2+2k})\\
\cdot & \psi(\sum_{l=n_1+2kn_2-k}^{n_1+2kn_2-1} w_{l,l+2k+1+2m})
\psi(a_1 w_{n_1+2kn_2,n_1+2kn_2+2k+1+2m})\\
\cdot & \psi_{[p_1^{e_1}\cdots p_s^{e_s}],\ul{a}}(\omega_1^{-1} v' \omega_1).
\end{split}
\end{align}

Let $i_0$ be the first row of $q_1$ such that $q_1(i_0,j) \equiv 0$ for some $j$. Assume that the number of $j$'s in the row $i_0 \leq i \leq n_1+2n_2k-2$ such that $q_1(i,j)\equiv 0$ is $s_i$.
For $i_0 \leq i \leq n_1+2n_2k-2$, let $X_j$'s,
$
(\sum_{k=i_0}^{i-1} s_k) +1 \leq j \leq \sum_{k=i_0}^{i} s_k,
$
be the one-dimensional subgroups such that the corresponding entries are in the $i$-th row of $q_1$ must be zero, from right to left.

For $i_0 \leq i \leq n_1+2n_2k-2$, let $Y_j$'s,
$(\sum_{k=i_0}^{i-1} s_k) +1 \leq j \leq \sum_{k=i_0}^{i} s_k,$
be the one-dimensional subgroups such that the corresponding entries are in the $(i+k+2)$-th column of $p_1$ not always zero, from bottom to top.
Write $\wt{W}_5$ be the subgroup of $W_5$ with $p_1$ part being zero, and denote $\psi_{\wt{W}_5} = \psi_{W_5}|_{\wt{W}_5}$.
Let $\ol{s} = \sum_{k=i_0}^{n_1+2n_2k-2} s_k$.

Consider the following sequence of quadruples
\begin{align}\label{lem3equ19}
\begin{split}
&(\prod_{i=2}^{\ol{s}} Y_{i} \wt{W}_5, \psi_{\wt{W}_5}, X_1, Y_1),\\
&(X_{1} \prod_{i=3}^{\ol{s}} Y_{i} \wt{W}_5, \psi_{\wt{W}_5}, X_{2}, Y_2),\\
&\cdots\\
&(\prod_{j=1}^{\ol{s}-1} X_{j} \wt{W}_5, \psi_{\wt{W}_5}, X_{\ol{s}}, Y_{\ol{s}}).
\end{split}
\end{align}
Applying Lemma \ref{nvequ} repeatedly, the integral
in \eqref{lem3equ9} is non-vanishing if and only if the
following integral is non-vanishing
\begin{equation}\label{lem3equ16}
\int_{[\prod_{j=1}^{\ol{s}} X_{j} \wt{W}_5]} \varphi(xw \omega_3 g) \psi_{\wt{W}_5}^{-1}(w) dxdw.
\end{equation}

Let $W_6=\prod_{j=1}^{\ol{s}} X_{j} \wt{W}_5$, and $\psi_{W_6}(xw) = \psi_{\wt{W}_5}(w)$.
Note that elements in $W_6$ have the following form
\begin{equation*}
w=\begin{pmatrix}
z & q_1 & q_2\\
0 & v' & q_1^*\\
0 & 0 & z^*
\end{pmatrix},
\end{equation*}
where the structures of $z, q_2, v'$ are the same as in \eqref{lem3equ17};
$$q_1 \in \Mat_{(n_1+k(2n_2+1)) \times (2m)}$$
such that no entries in the first $(n_1+2n_2k-2)$ rows are always zero.

If $q_1(n_1+2n_2k-1, 1:m)$ has entries that are not always zero, let $\alpha_i=e_{n_1+2n_2k-1}-e_{s_i}$, $1 \leq i \leq \ell$, $n-m+1 \leq s_i \leq n$, be all the roots such that the corresponding root subgroups in $q_1$ are not identically zero.
Let $\beta_i = e_{s_i}+e_{n_1+2n_2k+k}$, for $1 \leq i \leq \ell$.
Let $X = \prod_{i=1}^{\ell} X_{\alpha_i}$, $Y=\prod_{i=1}^{\ell} Y_{\beta_i}$.
Write $W_6 = X \wt{W}_6$ with $X \cap \wt{W}_6=\{I_{2n}\}$.
Let $\psi_{\wt{W}_6}=\psi_{{W_6}}|_{\wt{W}_6}$.
Consider the quadruple
$(\wt{W}_6, \psi_{\wt{W}_6}, X, Y)$.
By Lemma \ref{nvequ}, the integral
in \eqref{lem3equ16} is non-vanishing if and only if the
following integral is non-vanishing
\begin{equation}\label{lem3equ18}
\int_{[Y\wt{W}_6]} \varphi(yw \omega_3 g) \psi_{\wt{W}_6}^{-1}(w) dydw.
\end{equation}
Let $W_7=Y \wt{W}_6$, and $\psi_{W_7}(yw) = \psi_{\wt{W}_6}(w)$.


Define that
$A = \begin{pmatrix}
\frac{1}{2 \gamma} & 0 & \frac{1}{2}\\
0 & 1 & 0\\
-\frac{1}{2} & 0 & \frac{\gamma}{2}
\end{pmatrix}$,
$B = \diag(I_{2k-2}, A)$, and
$$C= \diag(I_{n_1-n_2+1}, B, B, \ldots, B),$$
with $n_2$-copies of $B$.
Let $\epsilon = (C, I_{2m+2k-2}, C^*)$.
It is easy to see that $\epsilon$ normalizes
$W_7$. Conjugating by $\epsilon$ from left, the integral in
\eqref{lem3equ9} becomes:
\begin{equation}\label{lem3equ12}
\int_{[W_8]} \varphi(w \epsilon \omega_3 g) \psi_{W_8}^{-1}(w) dw,
\end{equation}
where $W_8= \epsilon W_7 \epsilon^{-1}$,
$\psi_{W_8}(w)=\psi_{W_7}(\epsilon^{-1} w \epsilon)$.
It is easy to see that
\begin{align}\label{lem3equ13}
\begin{split}
&\psi_{W_8}(w)\\
=  \ & \psi(\sum_{i=1}^{n_1-n_2-1} w_{i,i+1})
\psi(\sum_{j=n_1-n_2+2}^{n_1+2kn_2-k-1} w_{j,j+2k+1})\\
\cdot &\psi(w_{n_1-n_2, n_1-n_2+2k+1})\psi(w_{n_1-n_2+1, n_1-n_2+2k})
\psi(\sum_{l=n_1+2kn_2-k}^{n_1+2kn_2-2} w_{l,l+2k+1+2m})\\
\cdot &
\psi(\frac{1}{\gamma} w_{n_1+2kn_2-1,n_1+2kn_2+2k+2+2m})\psi(-\gamma w_{n_1+2kn_2+1,n_1+2kn_2+2k+2m})\\
\cdot & \psi(a_1 w_{n_1+2kn_2,n_1+2kn_2+2k+1+2m})
\psi_{[p_1^{e_1}\cdots p_s^{e_s}],\ul{a}}(\omega_1^{-1} v' \omega_1).
\end{split}
\end{align}

Let $\ul{p}' = [(2n_1)(2n_2+2)(2n_2+1)^{2k-2}(2n_2)p_{1}^{e_1} \cdots p_{s}^{e_s}]$.
Now, it is easy to see that there is a subgroup $Y'$ of $Y''$,
with $X'' \oplus Y''$ being a polarization
of $V_{\ul{p}', 1} / V_{\ul{p}', 2}$ and preserving the character
$\psi_{\ul{p}', \{a_1, \frac{1}{2}, -\frac{1}{2}\} \cup \ul{a}}$, such that after conjugating by a certain Weyl element
$\omega_4$ from the left, and considering similar quadruples as in \eqref{lem3equ19}, the following integral
is non-vanishing
\begin{equation}\label{lem3equ14}
\int_{[Y'V_{\ul{p}', 2}]} \varphi(vyg) \psi_{\ul{p}', \{a_1, \frac{1}{2}, -\frac{1}{2}\} \cup \ul{a}}^{-1}(v) dvdy,
\end{equation}
if and only if the integral in \eqref{lem3equ12} is non-vanishing.
By Corollary \ref{halfheisen}, the integral in \eqref{lem3equ14} is
non-vanishing if and only if the following integral is non-vanishing
\begin{equation}\label{lem3equ15}
\int_{[V_{\ul{p}', 2}]} \varphi(vg) \psi_{\ul{p}', \{a_1, \frac{1}{2}, -\frac{1}{2}\} \cup \ul{a}}^{-1}(v) dv,
\end{equation}
which is a Fourier coefficient attached to the partition
$$
\ul{p}' = [(2n_1)(2n_2+2)(2n_2+1)^{2k-2}(2n_2)p_{1}^{e_1} \cdots p_{s}^{e_s}].
$$
Therefore, $\pi$ has a nonzero Fourier coefficient
attached to the expected partition $[(2n_1)(2n_2+2)(2n_2+1)^{2k-2}(2n_2)p_{1}^{e_1} \cdots p_{s}^{e_s}]$.
This completes the proof of the lemma.
\end{proof}


\section{Refined Properties of Fourier Coefficients}

In this section, we assume that $\pi$ is an irreducible
cuspidal automorphic representation of $G_n(\BA)$. Theorem \ref{thm9} shows that every partition $\udl{p}\in\frak{n}^m(\pi)$ is special for $G_n(\BA) = \Sp_{2n}(\BA)$.
It is also generally believed that for any $\udl{p}\in\frak{n}^m(\pi)$ and for $\udl{a}$ such that $\pi$ has a nonzero
$\psi_{\udl{p},\udl{a}}$-Fourier coefficient, the stabilizer $\Stab_{L_{\udl{p}}}(\psi_{\udl{p},\udl{a}})$ should be an $F$-anisotropic
algebraic group. There is currently no direct global evidence to support this proposition, except a conjecture given by Ginzburg
\cite[Conjecture 4.3]{G06}. On the other hand, over a local field corresponding to each local place of $F$,
the work of M{\oe}glin and Waldspurger (\cite{MW87}) and of M{\oe}glin (\cite{M96}) for $p$-adic local fields and the main
result of B. Harris's PhD thesis under supervision of D. Vogan (\cite{H11}) for archimedean local fields strongly support this
conjecture.

We study here this conjecture for $G_n(\BA)=\Sp_{2n}(\BA)$ or $\wt{\Sp}_{2n}(\BA)$. It is clear that if one knows more about the set $\frak{n}^m(\pi)$, it would be easier
to understand the conjecture. As suggested from \cite{Ka87}, \cite{MW87}, \cite{M98}, and \cite{GRS03}, the set $\frak{n}^m(\pi)$
should contain only one partition. This is one of the problems concerning Fourier coefficients that we are going to study in our future work.
On the other hand, Ginzburg, Rallis and Soudry found in \cite{GRS03} a particular partition in $\frak{n}^m(\pi)$, which is denoted by
$\ul{p}(\pi)$ for all irreducible cuspidal automorphic representations $\pi$ of $G_n(\BA)$. This particular partition
$\ul{p}(\pi)$ has the properties that
\begin{enumerate}
\item[(1)] every part in $\ul{p}(\pi)$ is even, and
\item[(2)] it is ``maximal at every stage" in the sense of the proof of
Theorem 2.7 of \cite{GRS03}.
\end{enumerate}
The property that $\ul{p}(\pi)\in \frak{n}^m(\pi)$ is ``maximal at every stage" will also play a crucial role in our investigation of the
conjecture below. In the following remark, we briefly explain the meaning of being ``maximal at every stage".

\begin{rmk}\label{rmk2}
For any irreducible automorphic representation $\pi$ of $G_n(\BA)$,
the descent of $\pi$ with respect to
the partition $[(2n_1)1^{2n-2n_1}]$ and
character $\psi_{[(2n_1)1^{2n-2n_1}], a_1}$ is defined on page 4 of \cite{GRS03}, and Section 3.3 of \cite{GRS11}. Denote it by $FJ_{\psi_{[(2n_1)1^{2n-2n_1}], a_1}}(\pi)$.
Then by Lemma \ref{lem1},
$\pi$ has a nonzero $\psi_{\ul{p}, \{a_1\} \cup \ul{a}}$-Fourier
coefficient attached to $\ul{p}=[(2n_1)p_2^{e_2}\cdots p_r^{e_r}]$
if and only if $FJ_{\psi_{[(2n_1)1^{2n-2n_1}], a_1}}(\pi)$ is
nonzero and it has a nonzero $\psi_{[p_2^{e_2}\cdots p_r^{e_r}], \ul{a}}$-Fourier
coefficient attached to the partition $[p_2^{e_2}\cdots p_r^{e_r}]$.

Note that if $\pi$ is also cuspidal,
and $\ul{p}(\pi)=[(2n_1)p_2\cdots p_r] \in \mathfrak{n}^m(\pi)$,
then Lemma \ref{lem1} implies that
$FJ_{\psi_{[(2n_1)1^{2n-2n_1}], a_1}}(\pi)$
is a nonzero and has a nonzero $\psi_{[p_2\cdots p_r], \ul{a}}$-Fourier
coefficient attached to the partition $[p_2\cdots p_r]$.
Moreover, by Lemma 2.3 of \cite{GRS03} and the property of ``maximal at evey stage" of $\ul{p}(\pi)$ (see the proof of Theorem 2.7 of \cite{GRS03}),
$FJ_{\psi_{[(2n_1)1^{2n-2n_1}], a_1}}(\pi)$
is a cuspidal representation of $\widetilde{G}_{n-n_1}(\BA)$ with
\begin{equation}\label{sec3equ1}
\ul{p}(FJ_{\psi_{[(2n_1)1^{2n-2n_1}], a_1}}(\pi)) =
[p_2\cdots p_r] \in \mathfrak{n}^m(FJ_{\psi_{[(2n_1)1^{2n-2n_1}], a_1}}(\pi)),
\end{equation}
where $\widetilde{G}_{n-n_1}(\BA) = \tilsp_{2n-2n_1}(\BA)$ if
$G_n(\BA) = \Sp_{2n}(\BA)$, $\widetilde{G}_{n-n_1}(\BA) = \Sp_{2n-2n_1}(\BA)$ if
$G_n = \tilsp_{2n}(\BA)$.
\end{rmk}

If we assume that the set $\frak{n}^m(\pi)$ contains only one partition, then we only need to prove that
the stabilizer $\Stab_{L_{\udl{p}}}(\psi_{\udl{p},\udl{a}})$ is $F$-anisotropic for $\udl{p}=\udl{p}(\pi)$.
\begin{thm}\label{main}
Let $\pi$ be an irreducible cuspidal automorphic representation of $\Sp_{2n}(\BA)$ or $\wt{\Sp}_{2n}(\BA)$. Assume that the set $\frak{n}^m(\pi)$ contains only one partition.
Then for any data $\udl{a}$ such that $\pi$ has a nonzero $\psi_{\udl{p},\udl{a}}$-Fourier coefficient with $\udl{p}\in \frak{n}^m(\pi)$,
the stabilizer $\Stab_{L_{\udl{p}}}(\psi_{\udl{p},\udl{a}})$ is $F$-anisotropic.
\end{thm}

\begin{proof}
We may write an even symplectic partition as constructed in \cite{GRS03} as
$$
\udl{p}(\pi)=[(2n_1)^{s_1}(2n_2)^{s_2}\cdots(2n_k)^{s_k}] \in \mathfrak{n}^m(\pi_1)
$$
with $n_1 > n_2 > \cdots > n_k$, $s_i \geq 1$. Then the stabilizer of
$\psi_{\underline{p}, \underline{a}}$ in $L_{\underline{p}}$
is isomorphic to
$$
O_1 \times O_2 \times \cdots \times O_k,
$$
where $O_j$ is the orthogonal group with respect to
the quadratic form $\diag(a_{\sum_{i=1}^{j-1}s_i +1}, a_{\sum_{i=1}^{j-1}s_i +2},
\ldots, a_{\sum_{i=1}^{j}s_i})$ for $1 \leq j \leq k$. It is enough to show that
for $\udl{p}=\udl{p}(\pi)$, all the quadratic forms
\begin{equation}
\diag(a_{\sum_{i=1}^{j-1}s_i +1}, a_{\sum_{i=1}^{j-1}s_i +2},
\ldots, a_{\sum_{i=1}^{j}s_i})
\end{equation}
for $1 \leq j \leq k$ are $F$-anisotropic.

We will prove this last statement by contradiction.
First, we claim that without loss of generality, we
may assume that the quadratic form
$$
\diag(a_{1}, a_{2},
\ldots, a_{s_1})
$$
is $F$-isotropic.

Indeed assume that
there exists $1 < j \leq k$ such that the quadratic form
$$
\diag(a_{\sum_{i=1}^{j-1}s_i +1}, a_{\sum_{i=1}^{j-1}s_i +2},
\ldots, a_{\sum_{i=1}^{j}s_i})
$$
is $F$-isotropic.
Let $\pi_1 = FJ_{\psi_{[(2n_1)1^{2n-2n_1}], a_1}}(\pi)$ be
the descent of $\pi$ with respect to
the partition $[(2n_1)1^{2n-2n_1}]$ and
character $\psi_{[(2n_1)1^{2n-2n_1}], a_1}$ as in \cite{GRS11}.
By assumption, $\pi$ has a nonzero
$\psi_{\underline{p}(\pi), \underline{a}}$-Fourier coefficient.
From Lemma \ref{lem1} and Remark \ref{rmk2},
$\pi_1$ is cuspidal and nonzero, and
\begin{equation}\label{sec4equ1}
\ul{p}(\pi_1) = [(2n_1)^{s_1-1}(2n_2)^{s_2}\cdots(2n_k)^{s_k}] \in \mathfrak{n}^m(\pi_1).
\end{equation}

We repeat the above procedure of doing descent and get a sequence of
nonzero cuspidal representations (assuming that $s_1 \geq 2$):
\begin{align}\label{sec4equ2}
\begin{split}
\pi_2 & = FJ_{\psi_{[(2n_1)1^{2n-4n_1}], a_2}}(\pi_1),\\
& \cdots\\
\pi_{\sum_{i=1}^{j-1}s_i} & =
FJ_{\psi_{[(2n_{j-1})1^{2n-2\sum_{i=1}^{j-1} n_is_i}], a_2}}(\pi_{\sum_{i=1}^{j-1}s_i-1}),
\end{split}
\end{align}
with
\begin{equation}\label{sec4equ3}
\ul{p}(\pi_{\sum_{i=1}^{j-1}s_i}) = [(2n_j)^{s_j}(2n_{j+1})^{s_{j+1}}\cdots(2n_k)^{s_k}]
\in \mathfrak{n}^m(\pi_{\sum_{i=1}^{j-1}s_i}).
\end{equation}
Then, we can start to prove the theorem by considering
$\pi_{\sum_{i=1}^{j-1}s_i}$ which is cuspidal, with
$$
\ul{p}(\pi_{\sum_{i=1}^{j-1}s_i}) = [(2n_j)^{s_j}(2n_{j+1})^{s_{j+1}}\cdots(2n_k)^{s_k}]
\in \mathfrak{n}^m(\pi_{\sum_{i=1}^{j-1}s_i}).
$$
Note that if $s_1=1$, then
$$
\pi_2 = FJ_{\psi_{[(2n_2)1^{2n-2n_1-2n_2}], a_2}}(\pi_1).
$$
Therefore, we have proved the claim.

Now, we assume that the quadratic form
$$
\diag(a_{1}, a_{2},
\ldots, a_{s_1})
$$
is $F$-isotropic.
Note that this implies that $s_1 \geq 2$.

Since every isotropic quadratic space is universal and
contains a hyperbolic plane, we can assume that $a_1 = 1$ and
$a_2 = -1$. We are going to show that as a cuspidal representation,
$\pi$ has no nonzero
$\psi_{\ul{p}(\pi), \{\ul{a}\}}$-Fourier coefficients,
with $\ul{a} = \{1,-1,a_3, \ldots, a_r\}$, $a_i \in F^*/(F^*)^2$ and $r = \sum_{i=1}^k s_k$. On the other hand, by the assumption of the theorem,
$\pi$ has a nonzero
$\psi_{\ul{p}(\pi), \{\ul{a}\}}$-Fourier coefficient, which is a contradiction.

By Lemma \ref{lem1}, it is enough to
show that any
$\psi_{[(2n_1)^2 1^{2n-4n_1}], \{1,-1\}}$-Fourier coefficient of $\pi$,
attached to the partition $[(2n_1)^2 1^{2n-4n_1}]$, is identically zero.
Note that by the property of ``maximal at every stage"
of $\ul{p}(\pi)$, see the proof of Theorem 2.7 of \cite{GRS03},
for any $n \geq m > n_1$, $\pi$ has no nonzero
Fourier coefficients attached to the partition $[(2m)1^{2n-2m}]$.

For any $\varphi \in \pi$, its
$\psi_{[(2n_1)^2 1^{2n-4n_1}], \{1,-1\}}$-Fourier coefficient
is defined as follows
\begin{align}\label{sec4equ4}
\begin{split}
& \varphi^{\psi_{[(2n_1)^2 1^{2n-4n_1}], \{1,-1\}}}(g)\\
=  \ & \int_{[V_{[(2n_1)^2 1^{2n-4n_1}], 2}]} \varphi(vg)
\psi_{[(2n_1)^2 1^{2n-4n_1}], \{1,-1\}}^{-1}(v) dv.
\end{split}
\end{align}

By Corollary \ref{halfheisen},
the $\psi_{[(2n_1)^2 1^{2n-4n_1}], \{1,-1\}}$-Fourier coefficient
of $\varphi$ in \eqref{sec4equ4} is non-vanishing
if and only if the following integral
is non-vanishing:
\begin{align}\label{sec4equ5}
\begin{split}
& \varphi^{\psi_{[(2n_1)^2 1^{2n-4n_1}], \{1,-1\}}}(g)\\
=  \ & \int_{[YV_{[(2n_1)^2 1^{2n-4n_1}], 2}]} \varphi(vyg)
\psi_{[(2n_1)^2 1^{2n-4n_1}], \{1,-1\}}^{-1}(v) dvdy,
\end{split}
\end{align}
where $Y$ is the group defined
in \eqref{equy} corresponding to
the partition $[(2n_1)^2 1^{2n-4n_1}]$.

Recall from \eqref{equtorus} that the one-dimensional toric
subgroup $\CH_{[(2n_1)^2 1^{2n-4n_1}]}$ corresponding to the partition
$[(2n_1)^2 1^{2n-4n_1}]$ has elements as follows
\begin{equation}\label{sec4equ6}
\mathcal{H}_{[(2n_1)^2 1^{2n-4n_1}]}(t):=
\diag(T_1; T_2; I_{2n-4n_1}; (T_2)^*; (T_1)^*),
\end{equation}
where
$$T_1 = \diag(t_1^{2n_1-1}, t_1^{2n_1-3}, \ldots, t_1),$$
$$T_2 = \diag(t_2^{2n_1-1}, t_2^{2n_1-3}, \ldots, t_2),$$
and actually $t_1 = t_2 = t$, we just label them to distinguish
the positions.

Let $\omega_1$ be a Weyl element sending the one-dimensional toric
subgroup $\CH_{[(2n_1)^2 1^{2n-4n_1}]}$ above to the following
one-dimensional toric subgroup
\begin{equation}\label{sec4equ7}
\widetilde{\mathcal{H}}_{[(2n_1)^2 1^{2n-4n_1}]}:=
\{\diag(\widetilde{T}_1; I_{2n-4n_1}; (\widetilde{T}_1)^*)\},
\end{equation}
where
$$\widetilde{T}_1 = \diag(t_1^{2n_1-1}, t_2^{2n_1-1};
t_1^{2n_1-3}, t_2^{2n_1-3}; \ldots; t_1, t_2).$$

Conjugating by $\omega_1$,
the integral in \eqref{sec4equ5} becomes
\begin{align}\label{sec4equ8}
\begin{split}
\int_{[U]} \varphi(u \omega_1 g)
\psi_{U}^{-1}(u) du,
\end{split}
\end{align}
where
\begin{eqnarray*}
U &=& \omega_1 YV_{[(2n_1)^2 1^{2n-4n_1}],2} \omega_1^{-1},\\
\psi_U(u)&=&\psi_{[(2n_1)^2 1^{2n-4n_1}], \{1,-1\}}(\omega_1^{-1} u \omega_1),
\end{eqnarray*}
and elements in $U$ have the following form
\begin{equation}\label{sec4equ9}
\begin{pmatrix}
z & q_1 & q_2\\
0 & I_{2n-4n_1} & q_1^*\\
0 & 0 & z^*
\end{pmatrix}
\end{equation}
where $z$ in unipotent radical of the parabolic
subgroup of $\GL_{2n_1}$ with Levi isomorphic to
$\GL_2 \times \GL_2 \times \cdots \times \GL_2$;
$q_1 \in \Mat_{(2n_1) \times (2n-4n_1)}$ with
$q_1(i,j)=0$ for $i=2n_1-1, 2n_1$ and $1 \leq j \leq n-2n_1$;
$q_2 \in \Mat_{(2n_1) \times (2n_1)}'$, which consists of
matrices $q$ in $\Mat_{(2n_1) \times (2n_1)}$ satisfying the property
that $q^t v_{2n_1} - v_{2n_1} q=0$, where $v_{2n_1}$ is a matrix only with
ones on the second diagonal.
The character is given by
\begin{align}\label{sec4equ10}
\begin{split}
& \psi_U\left(\begin{pmatrix}
z & q_1 & q_2\\
0 & I_{2n-4n_1} & q_1^*\\
0 & 0 & z^*
\end{pmatrix}\right)\\
=  \ & \psi(\sum_{i=1}^{2n_1-2} z_{i, i+2} + q_2(2n_1-1,2) - q_2(2n_1, 1)).
\end{split}
\end{align}

Let $A = \begin{pmatrix}
1 & -1\\
1 & 1
\end{pmatrix}$, and $\epsilon = \diag(1, A, \ldots, A, I_{2n-4k-2}, A^*, \ldots, A^*, 1)$.
Conjugating by $\epsilon$, the integral
in \eqref{sec4equ8} becomes:
\begin{equation}\label{sec4equ11}
\int_{[U^{\epsilon}]}
\varphi(u \epsilon \omega_1 g) \psi_{U^{\epsilon}}^{-1}(u)
du,
\end{equation}
where $U^{\epsilon}:= \epsilon U \epsilon^{-1}$, and
for $u \in U^{\epsilon}$, $\psi_{U^{\epsilon}}(u) = \psi_U(\epsilon^{-1} u \epsilon)$.

Elements $u \in U^{\epsilon}$ have the following form:
\begin{equation}\label{sec4equ12}
\begin{pmatrix}
z & q_1 & q_2\\
0 & I_{2n-4n_1} & q_1^*\\
0 & 0 & z^*
\end{pmatrix},
\end{equation}
with the same structure as elements in \eqref{sec4equ10}.
And
\begin{align}\label{sec4equ13}
\begin{split}
& \psi_{U^{\epsilon}}\left(\begin{pmatrix}
z & q_1 & q_2\\
0 & I_{2n-4n_1} & q_1^*\\
0 & 0 & z^*
\end{pmatrix}\right) \\
=  \ & \psi(\sum_{i=1}^{2n_1-2} z_{i, i+2} + q_2(2n_1-1,1))).
\end{split}
\end{align}

Let $\omega_2$ be a Weyl element which sends
the one-dimensional toric subgroup in \eqref{sec4equ7}
to the following one-dimensional toric subgroup:
\begin{equation}\label{sec4equ14}
\{\diag(T'; I_{2n-4n_1}; T'')\},
\end{equation}
where
$$T'=\diag(t_1^{2n_1-1}, t_1^{2n_1-3}, \ldots, t_1, t_2^{-1},
t_2^{-3}, \ldots, t_2^{1-2n_1}),$$
and
$$T''=\diag(t_2^{2n_1-1}, t_2^{2n_1-3}, \ldots, t_2, t_1^{-1},
t_1^{-3}, \ldots, t_1^{1-2n_1}).$$

Conjugating by $\omega_2$, the integral in \eqref{sec4equ11}
becomes:
\begin{equation}\label{sec4equ15}
\int_{[W]}
\varphi(w \omega_2 \epsilon \omega_1 g) \psi_{W}^{-1}(w)
dw,
\end{equation}
where $W:= \omega_2 U^{\epsilon} \omega_2^{-1}$, and
for $w \in W$, $\psi_{W}(w) = \psi_{U^{\epsilon}}(\omega_2^{-1} w \omega_2)$.

Elements $w \in W$ have the following form:
\begin{equation}\label{sec4equ16}
\begin{pmatrix}
z & q_1 & q_2\\
0 & I_{2n-4n_1} & q_1^*\\
0 & 0 & z^*
\end{pmatrix}
\begin{pmatrix}
I_{2n_1} & 0 & 0\\
p_1 & I_{2n-4n_1} & 0\\
p_2 & p_1^* & I_{2n_1}
\end{pmatrix},
\end{equation}
where $z\in N_{2n_1}$; $q_1 \in \Mat_{(2n_1) \times (2n-4n_1)}$ and
$q_1(i,j)=0$ for all $i \geq n_1 +1$
or for $i=n_1$ and $1 \leq j \leq n-2n_1$;
$q_2, p_2 \in \Mat_{(2n_1) \times (2n_1)}'$ and
$q_2(i,j)=p_2(i,j)=0$ for $i \geq j$; and
$p_1 \in \Mat_{(2n-4n_1) \times (2n_1)}$ and
$p_1(i,j)=0$ for all $j \leq n_1$
or for $j=n_1+1$ and $n-2n_1+1 \leq i \leq 2n-4n_1$.
After conjugating by a suitable toric element to adjust the signs, the character $\psi_W$ can be given by
\begin{align}\label{sec4equ17}
\begin{split}
\psi_{W}\left(\begin{pmatrix}
z & q_1 & q_2\\
0 & I_{2n-4n_1} & q_1^*\\
0 & 0 & z^*
\end{pmatrix}\right)
=  \ & \psi(\sum_{i=1}^{2n_1-1} z_{i,i+1}).
\end{split}
\end{align}

Let $R = \prod_{i=1}^{n-2n_1} X_{\alpha_i}$,
where $X_{\alpha_i}$ is the one-dimensional root subgroup
corresponding to the root $\alpha_i = e_{n_1} - e_{2n_1+i}$,
for $1 \leq i \leq n-2n_1$.
Let $C = \prod_{i=1}^{n-2n_1} X_{\beta_i}$,
where $X_{\beta_i}$ is the one-dimensional root subgroup
corresponding to the root $\beta_i = e_{2n_1+i} - e_{n_1+1}$,
for $1 \leq i \leq n-2n_1$.

For $1 \leq i \leq n_1$, $1 \leq j \leq i$, define that
$\gamma_{i,j} = e_i + e_{2n_1 - i+j}$, and
$\delta_{i,j} = - e_{2n_1 - i+j} - e_{i+1}$.
Let $X_{\gamma_{i,j}}$, $X_{\delta_{i,j}}$
be the corresponding one-dimensional root subgroups.
Let $W = C \prod_{i=1}^{n_1} \prod_{j=1}^i X_{\delta_{i,j}} \widetilde{W}$,
with $C \prod_{i=1}^{n_1} \prod_{j=1}^i X_{\delta_{i,j}} \cap \widetilde{W} = \{1\}$,
and $\psi_{\widetilde{W}} = \psi_W|_{\widetilde{W}}$.

First, consider the following quadruple:
\begin{equation}\label{sec4equ18}
(\prod_{i=1}^{n_1} \prod_{j=1}^i X_{\delta_{i,j}} \widetilde{W},
\psi_{\widetilde{W}}, R, C).
\end{equation}
It is easy to see that it satisfies all the conditions
Lemma \ref{nvequ}. Therefore, by
Lemma \ref{nvequ}, the integral in
\eqref{sec4equ15} is non-vanishing if and only if
the following integral is non-vanishing:
\begin{equation}\label{sec4equ19}
\int_{[R\prod_{i=1}^{n_1} \prod_{j=1}^i X_{\delta_{i,j}} \widetilde{W}]}
\varphi(wx\omega_2 \epsilon \omega_1 g) \psi_{\widetilde{W}}^{-1}(w)
dw dx.
\end{equation}

Then, for $i=1$, we consider the following quadruple:
\begin{equation}\label{sec4equ20}
(R\prod_{i=2}^{n_1} \prod_{j=1}^i X_{\delta_{i,j}} \widetilde{W},
\psi_{\widetilde{W}}, X_{\gamma_{1,1}}, X_{\delta_{1,1}}),
\end{equation}
which also satisfies all the conditions of Lemma \ref{nvequ}.
By Lemma \ref{nvequ}, the integral in
\eqref{sec4equ19} is non-vanishing if and only if the following
integral is non-vanishing:
\begin{equation}\label{sec4equ21}
\int_{[R X_{\gamma_{1,1}}\prod_{i=2}^{n_1} \prod_{j=1}^i X_{\delta_{i,j}} \widetilde{W}]}
\varphi(wx\omega_2 \epsilon \omega_1 g) \psi_{\widetilde{W}}^{-1}(w)
dw dx.
\end{equation}

Next, we repeat the above procedure for $i=2, 3, \ldots, n_1$,
and for each $2 \leq i \leq n_1$, we consider the following sequence of
quadruples:
\begin{align*}
& (R\prod_{k=1}^{i-1} \prod_{j=1}^k X_{\gamma_{k,j}}\prod_{k=i+1}^{n_1} \prod_{j=2}^k X_{\delta_{k,j}} \widetilde{W},
\psi_{\widetilde{W}}, X_{\gamma_{i,1}}, X_{\delta_{i,1}}),\\
& (R X_{\gamma_{i,1}}\prod_{k=1}^{i-1} \prod_{j=1}^k X_{\gamma_{k,j}}\prod_{k=i+1}^{n_1} \prod_{j=2}^k X_{\delta_{k,j}} \widetilde{W},
\psi_{\widetilde{W}}, X_{\gamma_{i,2}}, X_{\delta_{i,2}}),\\
& \cdots,\\
& (R \prod_{j=1}^{i-1}X_{\gamma_{i,j}}\prod_{k=1}^{i-1} \prod_{j=1}^k X_{\gamma_{k,j}}\prod_{k=i+1}^{n_1} \widetilde{W},
\psi_{\widetilde{W}}, X_{\gamma_{i,i}}, X_{\delta_{i,i}}).
\end{align*}
After applying Lemma \ref{nvequ} to all these
quadruples one-by-one, the integral in
\eqref{sec4equ21} is
non-vanishing if and only if the following
integral is non-vanishing:
\begin{equation}\label{sec4equ22}
\int_{[R \prod_{i=1}^{n_1} \prod_{j=1}^i X_{\gamma_{i,j}} \widetilde{W}]}
\varphi(wx\omega_2 \epsilon \omega_1 g) \psi_{\widetilde{W}}^{-1}(w)
dw dx.
\end{equation}

It is easy to see that the integral in \eqref{sec4equ22}
has an inner integral of the following form:
\begin{equation}\label{sec4equ23}
\int_{[U_{n_1}]} \varphi(ug) \psi_{U_{n_1}}^{-1}(u) du,
\end{equation}
where $U_{n_1}$ is the unipotent radical of the parabolic subgroup
of $G_n$ with Levi isomorphic to $\GL_1^{n_1} \times G_{n-n_1}$,
and
$$
\psi_{U_{n_1}}(u)=\psi(\sum_{i=1}^{n_1} u_{i,i+1}).
$$

By Lemma 2.2 of \cite{GRS03}, if the integral in \eqref{sec4equ23}
is non-vanishing, then there exists $n_1 < m \leq n$, such that
$\pi$ has a non-zero Fourier coefficient attached to the
partition $[(2m)1^{2n-2m}]$. But, this contradicts the
``maximal at every stage" property of $\ul{p}(\pi)$.
Therefore, the integral in \eqref{sec4equ23} is identically
zero. Hence, from the above discussion, we conclude that
$\pi$ has no nonzero
$\psi_{[(2n_1)^2 1^{2n-4n_1}], \{1,-1\}}$-Fourier coefficients
attached to the partition $[(2n_1)^2 1^{2n-4n_1}]$.
Then, by Lemma \ref{lem1}, $\pi$
has no nonzero
$\psi_{\ul{p}(\pi), \{\ul{a}\}}$-Fourier coefficients,
with $\ul{a} = \{1,-1,a_3, \ldots, a_r\}$ for all $a_i\in F^*/(F^*)^2$,
which contradicts the assumption of the Theorem. This completes the proof of the theorem.
\end{proof}

Theorem \ref{main} has the following application that the type of Fourier coefficients could be restricted when the
underground number field $F$ is totally imaginary. Further applications to the discrete spectrum of square-integrable automorphic forms
based on the Arthur classification will be addressed in our future work.

\begin{thm}\label{thm1}
Assume that the number field $F$ is totally imaginary. Let $\pi$ be an irreducible cuspidal automorphic
representation of $\Sp_{2n}(\BA)$, or $\tilsp_{2n}(\BA)$, with the even symplectic partition
constructed in Theorem 2.7 of \cite{GRS03}
$$
\ul{p}(\pi)=[(2n_1)^{s_1}(2n_2)^{s_2}\cdots(2n_k)^{s_k}] \in \mathfrak{n}^m(\pi),
$$
where $n_1 > n_2 > \cdots > n_k$.
Then the inequality $s_i \leq 4$ holds for $i=1,2,\cdots, k$.
\end{thm}

\begin{proof}
Since $\ul{p}(\pi)=[(2n_1)^{s_1}(2n_2)^{s_2}\cdots(2n_k)^{s_k}] \in \mathfrak{n}^m(\pi)$,
$\pi$ must have a nonzero
$\psi_{\underline{p}(\pi), \underline{a}}$-Fourier coefficient
for some $\ul{a}=\{a_1, \ldots, a_r\}$ with all $a_i \in F^*/(F^*)^2$, and
$r = \sum_{i=1}^k s_i$.

Assume that there is some $1 \leq j \leq k$, such that $s_j \geq 5$.
Then, there is quadratic form
$$
\diag(a_{\sum_{i=1}^{j-1}s_i +1}, a_{\sum_{i=1}^{j-1}s_i +2},
\ldots, a_{\sum_{i=1}^{j}s_i}),
$$
with dimension $s_j \geq 5$.
Since $F$ is a totally imaginary number field,
it is well-known that any non-degenerate quadratic form of dimension bigger than or equal to $5$
is $F$-isotropic. Hence the quadratic form
$$
\diag(a_{\sum_{i=1}^{j-1}s_i +1}, a_{\sum_{i=1}^{j-1}s_i +2},
\ldots, a_{\sum_{i=1}^{j}s_i})
$$
is $F$-isotropic.
On the other hand, by Theorem \ref{main},
the quadratic form
$$
\diag(a_{\sum_{i=1}^{j-1}s_i +1}, a_{\sum_{i=1}^{j-1}s_i +2},
\ldots, a_{\sum_{i=1}^{j}s_i})
$$
is $F$-anisotropic for every $1 \leq j \leq k$. This contradiction implies that
for all $i=1,2,\cdots, k$, the inequality $s_i \leq 4$ must hold.
\end{proof}

\begin{rmk}
The assumption that $F$ is totally imaginary in Theorem \ref{thm1} is not merely for a technical reason. Without it 
the statement of the theorem may no longer be true.
\end{rmk}

\end{document}